\documentclass[12pt, reqno]{amsart}
\usepackage{amsmath, amsthm, amscd, amsfonts, amssymb, graphicx, color}
\usepackage[bookmarksnumbered, colorlinks, plainpages]{hyperref}
\hypersetup{colorlinks=true,linkcolor=red, anchorcolor=green, citecolor=cyan, urlcolor=red, filecolor=magenta, pdftoolbar=true}

\newtheorem{theorem}{Theorem}[section]
\newtheorem{lemma}[theorem]{Lemma}
\newtheorem{proposition}[theorem]{Proposition}
\newtheorem{corollary}[theorem]{Corollary}

\newtheorem*{ExampleCS}{Example}

\theoremstyle{definition}

\newtheorem{example}[theorem]{Example}

\theoremstyle{remark}

\numberwithin{equation}{section}

\newcommand\C{\mathbb C}
\newcommand\D{\mathbb D}

\newcommand\rimu{R^\infty(K,\mu)}
\newcommand\rikmu{R^\infty(K,\mu)}

\newcommand\limu{L^\infty(\mu)}
\renewcommand\i{\infty}
\newcommand\area{{\frak m}}
\newcommand\CT{{\mathcal C}}

\begin{document}

\setcounter{page}{1}

\title[Multiplication Operators]{Multiplication Operators on Hilbert Spaces}

\author[L. Yang]{Liming Yang$^1$}

\address{$^1$Department of Mathematics, Virginia Polytechnic and State University, Blacksburg, VA 24061.}
\email{\textcolor[rgb]{0.00,0.00,0.84}{yliming@vt.edu}}




\begin{abstract} Let $S$ be a subnormal operator on a separable complex Hilbert space $\mathcal H$ and let $\mu$ be the scalar-valued spectral measure for the minimal normal extension $N$ of $S.$ Let $R^\i (\sigma(S),\mu)$ be the weak-star closure in $L^\i(\mu)$ of rational functions with poles off $\sigma(S),$ the spectrum of $S.$ The multiplier algebra $M(S)$ consists of functions $f\in L^\infty(\mu)$ such that $f(N)\mathcal H \subset \mathcal H.$ The
multiplication operator $M_{S,f}$ of $f\in M(S)$ is defined $M_{S,f} = f(N) |_{\mathcal H}.$ We show that  for $f\in R^\i (\sigma(S),\mu),$ (1) $M_{S,f}$ is invertible iff $f$ is invertible in $M(S)$ and (2) $M_{S,f}$ is Fredholm iff there exists $f_0\in R^\i (\sigma(S),\mu)$ and a polynomial $p$ such that $f=pf_0,$ $f_0$ is invertible in $M(S),$ and $p$ has only zeros in $\sigma (S) \setminus \sigma_e (S),$ where $\sigma_e (S)$ denotes the essential spectrum of $S.$
Consequently, we characterize $\sigma(M_{S,f})$ and $\sigma_e(M_{S,f})$ in terms of some cluster subsets of $f.$ Moreover, we show that if $S$ is an irreducible subnormal operator and $f \in R^\i (\sigma(S),\mu),$ then $M_{S,f}$ is invertible iff $f$ is invertible in $R^\i (\sigma(S),\mu).$ The results answer the second open question raised by J. Dudziak in 1984.  
\end{abstract} \maketitle

\section{\textbf{Introduction}}

\subsection{Backgrounds} 
For $\mathcal H$ a complex separable Hilbert space, $\mathcal L(\mathcal H)$ denotes the space of all bounded linear operators on $\mathcal H.$ The spectrum and essential spectrum of an operator $T\in \mathcal L(\mathcal H)$ are denoted $\sigma(T)$ and $\sigma_e(T),$ respectively.

An operator $S\in \mathcal L(\mathcal H)$ is called {\em  subnormal} if there exists a complex separable Hilbert space $\mathcal K$ containing $\mathcal H$ and a normal operator $N\in \mathcal L(\mathcal K)$ such that $N\mathcal H \subset \mathcal H$ and $S = N|_{\mathcal H}.$ Such an $N$ is called a {\em minimal normal extension} (mne) of $S$ if the smallest closed subspace of $\mathcal K$ containing $\mathcal H$ and reducing $N$ is $\mathcal K$ itself. Any two mnes of $S$ are unitarily equivalent in a manner that fixes $S.$ We say $S$ is pure if $S$ does not contain a non-trivial direct normal summand and $S$ is irreducible if $S$ does not have a non-trivial reducing invariant subspace.
The spectrum $\sigma (S)$ is the union of $\sigma (N)$ and some collection of bounded components of $\C\setminus \sigma (N).$ The book \cite{conway} is a good reference for basic information for subnormal operators.

For a subnormal operator $S\in \mathcal L(\mathcal H)$ with $N=mne(S),$ let $\mu$ be the scalar-valued spectral measure (svsm) for $N.$ For a compact subset $K\supset \sigma(S),$ let $\mbox{Rat}(K)$ be the set of rational functions with poles off $K$
and let $\rikmu$ be the weak-star closure of $\mbox{Rat}(K)$ in $\limu.$
For $f\in L^\i(\mu),$ the normal operator $f(N)$ is well defined. The multiplier algebra of $S$ is defined
\[
\ M(S) = \{f\in L^\i(\mu):~ f(N)\mathcal H \subset \mathcal H\}.
\]
Since the support of $\mu$ is a subset of $\sigma(N) \subset  \sigma (S),$ we see 
\[
\ R^\i(\sigma (S), \mu) \subset M(S).
\]
The multiplication operator $M_{S,f} := f(N) |_{\mathcal H}$ defines a functional calculus for $S$ and $f\in M(S).$ 

For a compact subset $K\subset \C$ and a finite positive measure $\mu$ supported on $K,$ let $R^2(K,\mu)$ denote the closure of $\text{Rat}(K)$ in $L^2(\mu)$ norm and let $S_\mu$ denote the multiplication by $z$ on $R^2(K,\mu).$  

\begin{ExampleCS}
If $S$ is a rationally cyclic subnormal operator, then $S$ is unitarily equivalent to $S_\mu$ on $R^2(K,\mu),$ where $K = \sigma(S_\mu)$ (see \cite[III.5.2]{c81}). In this case, 
\[
\ M(S_\mu) = R^2(\sigma(S_\mu),\mu) \cap L^\i (\mu).
\]
\end{ExampleCS}

We concern the following questions for a subnormal operator $S$ and $f\in R^\i(\sigma (S), \mu)$ (or $f\in M(S)$):
\newline
(Q1) When is $M_{S,f}$ invertible?
\newline
(Q2) When is $M_{S,f}$ Fredholm?
\newline
(Q3) Is there a spectral mapping theorem for the function calculus for subnormal operators (see \cite[Chapter VII Section 3]{conway})? More generally, can $\sigma(M_{S,f})$ and $\sigma_e(M_{S,f})$ be characterized in terms of function theoretic descriptions of $f?$

The objective of the paper is to answer all above questions for $f\in R^\i(\sigma (S), \mu).$

\subsection{Known Results}
For a Borel subset $B\subset\C,$ let $M_0(B)$ denote the set of finite complex-valued Borel measures that are compactly supported in $B$ and let $M_0^+(B)$ be the set of positive measures in $M_0(B).$ The support of $\nu\in M_0(\C),$ $\text{spt}(\nu),$ is the smallest closed set that has full $|\nu|$ measure. For a Borel set $A\subset \C,$ $\nu_A$ denotes $\nu$ restricted to $A$ and $\overline {A}$ denotes the closure of the set $A.$  We use $\nu-a.a.$ for a property that holds everywhere except possibly on a set of $|\nu|$ zero. Let $\area$ denote the area (Lebesgue) measure on $\C.$

For a bounded open subset $\Omega \subset \C,$ let $L_a^2(\Omega)$ be the Bergman space of analytic functions $f$ on $\Omega$ such that $\int_\Omega |f|^2 d\area < \infty$ and let $H^\infty (\Omega)$ denote the algebra of bounded and analytic functions on $\Omega.$ The Bergman shift $B_z$ is defined $B_zf=zf$ for $f\in L_a^2(\Omega).$ Then $B_z$ is subnormal, $M(B_z) = H^\infty (\Omega),$ $\sigma (B_z) = \overline \Omega,$ and $\sigma _e (B_z) = \partial_e \Omega,$ where $\partial_e \Omega$ denotes the essential boundary of $\Omega$ (see \cite[on page 29]{a82}). For $f\in H^\infty (\Omega),$ let $cl(f, \Omega)$ be the closure of the set $f(\Omega)$ and let $cl_e (f, \Omega)$ be the set of points $a\in \C$ such that there exists $\{\lambda_n\} \subset \Omega$ satisfying $\lim_{n\rightarrow \infty}\lambda_n \in \partial_e \Omega$ and $\lim_{n\rightarrow \infty}f(\lambda_n) = a.$
S. Axler \cite[Proposition 2 and Theorem 23]{a82} proved that
\[
\ \sigma (M_{B_z,f}) = cl (f, \Omega) \text{ and } \sigma_e (M_{B_z,f}) = cl _e (f, \Omega).
\]
The result answered the questions (Q1), (Q2), and (Q3) for the Bergman shift $B_z.$
Some special cases of the above result include Coburn \cite{c73}, Janas \cite{j81}, McDonald \cite{m77}, and Axler, Conway, and McDonald \cite{acm82}.

 The Cauchy transform of $\nu \in M_0(\C)$ is defined by
\begin{eqnarray} \label{CTDefInit}
\ \mathcal C\nu (z) = \int \dfrac{1}{w - z} d\nu (w)
\end{eqnarray}
for all $z\in\mathbb {C}$ for which
\begin{eqnarray} \label{CTDefInit2}
\ \tilde \nu (z) := \int \frac{d|\nu|(w)}{|w-z|} < \infty .	
\end{eqnarray}
A standard application of the Fubini's
Theorem shows that $\mathcal C\nu \in L^s_{loc}(\mathbb {C} )$ for $ 0 < s < 2.$ In particular, it is
defined for $\area-a.a..$

Let $K$ be a compact subset and $\mu \in M_0^+(K).$ There exists a Borel partition $\{\Delta_0,\Delta_1\}$ of $\text{spt}(\mu)$ such that 
\begin{eqnarray}\label{RIDecomp}
\ \rikmu = L^\i (\mu_{\Delta_0}) \oplus R^\i (K,\mu_{\Delta_1}) 
\end{eqnarray}
where $R^\i (K,\mu_{\Delta_1})$ contains no non-trivial direct $L^\i$ summands (see \cite[Proposition 1.16 on page 281]{conway} ). Call $\rikmu$ pure if $\Delta_0 = \emptyset$ in \eqref{RIDecomp}.
For $W\subset L^\i(\mu),$ let $W^\perp$ be the set of $g\in L^1(\mu)$ such that $\int fgd\mu = 0$  for all $f\in W.$

The {\em envelope} $E(K,\mu)$ with respect to $K$ and $\mu\in M_0^+(K)$ is the set of points $\lambda \in  K$ such that there exists $g\in \rikmu^\perp$ satisfying
\begin{eqnarray} \label{EKMuDef}
\  \widetilde{g\mu}(\lambda) < \i \text{ and } \CT(g\mu)(\lambda) \ne 0.	
\end{eqnarray} 
The set $E(K,\mu)$ is an $\area$ measurable set (see Proposition \ref{EProp} (b)).
For $\lambda\in E(K,\mu),$ $g\in \rikmu^\perp$ satisfying \eqref{EKMuDef}, and $f\in R^\i (K,,\mu),$ set $\rho_{K,\mu}(f)(\lambda) = \frac{\CT(fg\mu)(\lambda)}{\CT(g\mu)(\lambda)}.$ Clearly $\rho_{K,\mu}(r)(\lambda) = r(\lambda)$ for $r\in \text{Rat}(K)$ since $\frac{r(z) - r(\lambda)}{z - \lambda} \in \text{Rat}(K).$ Hence, $\rho_{K,\mu}(f)(\lambda)$ is independent of the particular $g$ chosen. 
We thus have a map $\rho_{K,\mu}$ called Chaumat's map for $K$ and $\mu,$ which associates
to each function in $\rikmu$ a point function on $E(K,\mu).$ Chaumat's Theorem \cite{cha74} (also see \cite[page 288]{conway}) states: If $\rikmu$ is pure, then $\rho_{K,\mu}$ is an isometric isomorphism and a weak-star homeomorphism from $\rikmu$ onto $R^\i (\overline {E(K,\mu)}, \area_{E(K,\mu)}).$

Let $S\in \mathcal L(\mathcal H)$ be a subnormal operator and let $\mu$ be the scalar-valued spectral measure for $N=mne(S)$ on $\mathcal K.$ Assume that $R^\i(\sigma (S), \mu)$ is pure (i.e. $S$ is $R^\i$ pure). Clearly, if $S$ is pure, then $S$ is $R^\i$ pure. 
Denote 
\[
\ E_S^\i = E(\sigma (S),\mu) \text{ and } \rho_S = \rho_{\sigma (S),\mu}.
\]
For $f\in R^\i(\sigma (S), \mu)$ and an $\area$ measurable subset $E \subset E_S^\i,$ let $cl (\rho_S(f), E)$ be the closure of the set  $\{\rho_S(f)(\lambda):~ \lambda \in E\}$  and let $cl_e (\rho_S(f), E)$ be the set of points $a\in \C$ such that there exists $\{\lambda_n\} \subset E$ satisfying $\lim_{n\rightarrow \infty}\lambda_n \in \sigma _e(S)$ and $\lim_{n\rightarrow \infty}\rho_S(f)(\lambda_n) = a.$
J. Dudziak \cite{dud84} shown if a compact subset $K \supset \sigma (S)$ satisfies a sufficient strong condition to actually guarantee that functions in $R^\i(K, \mu)$ are bounded and analytic on some open set, then for $f\in R^\i(K, \mu)$ (notice that we may have $R^\i(K, \mu) \subsetneqq R^\i(\sigma (S), \mu)$),
 \begin{eqnarray} \label{DudziakEq1}
\ \sigma (M_{S,f}) = cl (\rho_S(f), E_S^\i)
\end{eqnarray}
and
\begin{eqnarray} \label{DudziakEq2}
\ \sigma_e (M_{B_z,f}) = cl_e (\rho_S(f), E_S^\i).
\end{eqnarray}
Let $P^\infty(\mu)$ be the weak-star closure of polynomials in $L^\infty(\mu).$ Then $P^\infty(\mu) = R^\i(K, \mu),$ when $K$ is a disk containing $\sigma (S).$ Therefore, \eqref{DudziakEq1} and \eqref{DudziakEq2} imply the spectral mapping theorem for $f\in P^\infty(\mu)$ obtained by Conway and Olin in 1977 (see \cite[Corollary 3.8 on page 326]{conway} or \cite{co77}). 

Recently, the author \cite[Theorem 1.1 and Corollary 1.2]{y23} proved the conjecture $\diamond$ posed by J. Dudziak in \cite[first open question on page 386]{dud84} below.

\begin{theorem}\label{InvTheoremY23} (Yang 2023)
Let $S\in \mathcal L(\mathcal H)$ be a subnormal operator and let $\mu$ be the scalar-valued spectral measure for $N=mne(S).$ Assume $S$ is $R^\i$ pure. If for $f\in R^\i(\sigma (S), \mu),$ there exists $\epsilon_f > 0$ such that 
\[
\ |\rho_S(f)(z)| \ge \epsilon_f,~ \area_{E_S^\infty}-a.a.,
\] 
then $f$ is invertible in $R^\i(\sigma (S), \mu)$ and $M_{S,f}$ is invertible.	Consequently, 
\[
\ \sigma (M_{S,f}) \subset cl(\rho_S(f), E_S^\i).
\]
\end{theorem}

\subsection{Main Results} The second open question posted in \cite[page 386]{dud84} is that of whether the inclusion  
\[
\ \sigma (M_{S,f}) \supset cl(\rho_S(f), E_S^\infty)
\]
holds for every $R^\i$ pure subnormal operator $S$ and every $f\in R^\i(\sigma (S), \mu).$ Our first result, Corollary \ref{ExampleCor} in Section 2, provides an example to show that the above inclusion fails in general. That is, there exists a $R^\i$ pure subnormal operator $S$ and there exists a function $f \in R^\i(\sigma(S),\mu)$ such that
\begin{eqnarray} \label{DudziakEq1Fail}
\ \sigma (M_{S,f}) \subsetneqq cl(\rho_S(f), E_S^\infty).
\end{eqnarray}
Naturally we may ask the following question, a second version of (Q3):
\newline
(Q3') Is there a subset $E \subset E_S^\infty$ such that for $f\in R^\i(\sigma (S), \mu),$
\[
\ \sigma (M_{S,f}) = cl(\rho_S(f), E) \text{ and }\sigma_e (M_{S,f}) = cl_e(\rho_S(f), E)?
\]

By the spectral theorem of normal operators, we assume 
\begin{eqnarray}\label{NEHSpace}
\ \mathcal K = \bigoplus _{i = 1}^m L^2 (\mu _i)
\end{eqnarray}
where $\mu_i\in M_0^+(\C),$ $\mu_1(= \mu) \gg \mu_2 \gg ... \gg\mu_m$ ($m$ may be $\infty$), and $N = M_z$ is the multiplication by $z$ on $\mathcal K.$ For $H = (h_1,...,h_m), ~G = (g_1,...,g_m)\in \mathcal K,$ define
\begin{eqnarray}\label{HGProduct}
\ \langle H(z),G(z)\rangle = \sum_{i=1}^m h_i(z) \overline{g_i(z)} \dfrac{d\mu_i}{d\mu}.
\end{eqnarray}
Then the inner product $(H,G) = \int \langle H(z),G(z)\rangle d\mu(z).$ Set $ \mu_{H,G} = \langle H(z),G(z)\rangle \mu.$

 The {\it envelope} $E_S$ of $S$ is the set of points $\lambda \in  \sigma(S)$ such that there exist $H \in \mathcal H$ and $G\in \mathcal K \ominus \mathcal H$ satisfying $ \widetilde{\mu_{H,G}}(\lambda) < \i$ and $\CT\mu_{H,G}(\lambda) \ne 0.$ Clearly, by \eqref{EKMuDef}, $E_S \subset E_S^\i$ since $\langle H(z),G(z)\rangle \in R^\i(\sigma(S), \mu)^\perp.$  
 From Pr1position \ref{ESProp}, we see that $E_S$ is an $\area$ measurable subset of $E_S^\i.$ Example \ref{ESExample} shows that there exists a pure subnormal operator $S$ satisfying
 \[
 \ E_S \subsetneqq E_S^\i \text{ and } \area(E_S^\i \setminus E_S) > 0.
 \]
 
 Theorem \ref{MTheorem} generalizes Theorem \ref{InvTheoremY23} and proves that
for a pure subnormal operator $S$ and $f\in R^\i(\sigma (S), \mu),$ $f$ is invertible in $M(S)$ if and only if there exists $\epsilon_f > 0$ such that 
\begin{eqnarray} \label{MTEq}
\ |\rho_S(f)(z)| \ge \epsilon_f,~ \area_{E_S}-a.a..
\end{eqnarray}
Consequently, if $f\in R^\i(\sigma (S), \mu)$ satisfies \eqref{MTEq}, then $M_{S,f}$ is invertible.

As an application of Theorem \ref{MTheorem}, Theorem \ref{theoremA} (1) answers (Q1) as the following: If $f\in R^\i(\sigma (S), \mu),$ then $M_{S,f}$ is invertible if and only if $f$ is invertible in $M(S).$
Theorem \ref{theoremA} (2) answers (Q2) as the following: If $f\in R^\i(\sigma (S), \mu),$ then 
 $M_{S,f}$ is Fredholm if and only if there exists $f_0\in R^\i(\sigma (S), \mu)$ and a polynomial $p$ such that $f=pf_0,$ $f_0$ is invertible in $M(S),$ and $p$ has only zeros in $\sigma (S) \setminus \sigma_e (S).$ Combining Corollary \ref{corollaryA} and Theorem \ref{InvInMA}, we prove Corollary \ref{corollaryB} which affirmatively answers (Q3) and (Q3').
 
Corollary \ref{ExampleCor} shows that \eqref{DudziakEq1Fail} holds for some $S = \oplus_{n=1}^\i S_n,$ where $S_n$ is irreducible. Corollary  \ref{corollaryC} shows that this phenomenon disappears when $S$ is irreducible, that is, both \eqref{DudziakEq1} and \eqref{DudziakEq2} hold if $S$ is an irreducible subnormal operator. The result shows that the second open question in \cite[page 386]{dud84} has an affirmative answer for irreducible subnormal operators.

\section{\textbf{An Example of a Pure Subnormal Operator Satisfying \eqref{DudziakEq1Fail}}}

The elementary properties of the envelope $E(K,\mu)$ of $\rikmu$ are listed in the following proposition.

\begin{proposition}\label{EProp} Let $\mu \in M_0^+(K)$ for some compact subset $K\subset \C.$ If $\rimu$ is pure, then the following properties hold.
\newline
(a) $E(K,\mu)$ is the set of weak-star continuous homomorphisms on $\rikmu$ (see \cite[Proposition VI.2.5]{conway}).
\newline
(b) $E(K,\mu)$ is a nonempty $\area$ measurable set with area density one at each of its points (see \cite[Proposition VI.2.8]{conway}).
\newline
(c) $\text{int}(E(K,\mu)) = \text{int}(\overline {E(K,\mu)})$ (see \cite[Proposition VI.3.9]{conway}).
\end{proposition} 

J. Dudziak \cite{dud84} conjectured that an appropriate place to look for an example satisfying \eqref{DudziakEq1Fail} would be among rationally cyclic $R^\i$ pure subnormal operators. 
 Our theorem below provides such an example.

\begin{theorem} \label{ExampleThm}
There exists $\mu \in M_0^+(\C)$ and a function $f \in R^\i(\sigma(S_\mu),\mu)$ such that $\sigma(S_\mu) = \text{spt}\mu,$ $R^\i(\sigma(S_\mu),\mu)$ is pure,
$|f| \ge \epsilon,~\mu-a.a.$ for some $\epsilon > 0,$ $f$ is invertible in $M(S_\mu),$ and $f$ is not invertible in $R^\i (\sigma(S_\mu),\mu).$
\end{theorem}

As an application, our result below shows that $S_\mu$ in Theorem \ref{ExampleThm}satisfies  \eqref{DudziakEq1Fail}.

\begin{corollary} \label{ExampleCor}
There exists $\mu \in M_0^+(\C)$ and there exists a function $f \in R^\i(\sigma(S_\mu),\mu)$ such that $S_\mu$ is $R^\i$ pure and 
$cl(\rho_{S_\mu}(f), E_{S_\mu}^\infty) \setminus \sigma(M_{S_\mu,f}) \ne \emptyset.$	
\end{corollary}

\begin{proof}
Let $\mu$ and $f$ be as in Theorem \ref{ExampleThm}. Then, by Theorem \ref{ExampleThm}, there exists $f_0\in M(S_\mu)$ such that $ff_0 = 1$ and $f$ is not invertible in $R^\i (\sigma(S_\mu),\mu).$ Therefore, $M_{S_\mu,f_0}M_{S_\mu,f} = M_{S_\mu,f}M_{S_\mu,f_0} = I$ which implies $0\notin \sigma(M_{S_\mu,f}).$ Using Theorem \ref{InvTheoremY23} and Proposition \ref{EProp} (b), we conclude that $0\in cl(\rho_{S_\mu}(f), E_{S_\mu}^\i).$
	\end{proof}

In this section, we focus on proving Theorem \ref{ExampleThm}.
For $\lambda\in \C$ and $r >\delta > 0,$ let $\D(\lambda,\delta) = \{|z - \lambda | < \delta\}$ and let $A(\lambda,\delta, r) = \{\delta \le |z - \lambda | \le r\}.$ Set $\D = \D(0,1).$

\begin{lemma} \label{K1Exist}
There exist $z_n\in \D(0, \frac 12),$ $r_n,\delta_n > 0,$  and rational functions $f_n(z)$ with only poles in $\{z_k\}_{1\le k \le n}$ for all $n\ge 1$ such that the following properties hold.
\newline
(1) For $n,m \ge 1$ and $n \ne m,$ $z_n \ne 0,$ $\overline{\D(z_n,r_n)} \subset \D(0,\frac 12),$ $\overline{\D(z_n,r_n)} \cap \overline{\D(z_m,r_m)}  = \emptyset,$ and 
\begin{eqnarray} \label{K1ExistEq1}
\ \frac {r_n}{2} \le\delta_n <r_n < 2^{-n}|z_n|.
\end{eqnarray}
(2) For $n \ge 1,$ $f_n(0) = 0,$ $\frac 12 < |f_n(z)| < 2$ for $z\in \cup_{k=0}^n A_k,$ where $A_0 = A(0,\frac 12, 1)$ and $A_k = A(0,\delta_k, r_k)$ for $1 \le k \le n.$
\newline
(3) $K_1 := \overline{\D(0,\frac 12)} \setminus \cup_{n=1}^\i \D(z_n,r_n)$ is a compact subset without interior,  $0\in K_1,$ the set $\{z_n\}_{n=1}^\i$ is dense in $K_1,$ and
\begin{eqnarray} \label{K1ExistEq2}
\ \sum_{n=1}^\i  \dfrac{r_n}{|z_n|} < \i.
\end{eqnarray} 
\end{lemma}

\begin{proof}
Let $\{\lambda_k\}\subset \D(0,\frac 12)\setminus \{0\}$ be a dense subset of $\D(0,\frac 12)$ and $|\lambda_1 | > \frac 13.$ 
We will inductively construct $z_n,$ $r_n,$ $\delta_n,$ and the functions $f_n$ for all $n \ge 1.$ Set $z_1 = \lambda_{k_1},$ where $k_1 = 1.$ Let
$f_1(z) = \frac 32 z.$ We choose $r_1, \delta_1$ small enough such that (1) and (2) hold. This finishes the construction for $n = 1.$

Suppose that we have constructed $z_n = \lambda_{k_n},~ r_n,~ \delta_n,$ and $f_n$ such that  $\lambda_j \in \cup_{k=1}^n \overline {\D(z_k,r_k)}$ for $1 \le j \le k_n$ and the properties (1) and (2) hold. We now construct $z_{n+1},$ $r_{n+1},$ $\delta_{n+1},$ and $f_{n+1}$ as the following. We find $k_{n+1} > k_n$ such that $\lambda_{k_n+1},...,\lambda_{k_{n+1}-1} \in \cup_{k=1}^n \overline {\D(z_k,r_k)}$ and $\lambda_{k_{n+1}} \notin \cup_{k=1}^n \overline {\D(z_k,r_k)}.$ Set $z_{n+1} = \lambda_{k_{n+1}}.$ Let $a_{n+1} = f_n(z_{n+1}).$ 

If $|a_{n+1}| > \frac 12,$ then set $f_{n+1}(z) = f_n(z).$ We choose $r_{n+1}$ small enough such that $|f_{n+1}(z)| > \frac 12$ for $z\in \overline{\D(z_{n+1},r_{n+1})}$ and set $\delta_{n+1} = \frac{r_{n+1}}{2}.$

Now we assume that $|a_{n+1}| \le  \frac 12.$ Let 
\[
\ b_{n+1}(z) = \dfrac 54 \dfrac{z}{|z_{n+1}|} \frac{r_{n+1}}{z-z_{n+1}}.
\] 
Let $min_n = \min_{z\in \cup_{k=0}^n A_k} |f_n(z)|$ and $max_n = \max_{z\in \cup_{k=0}^n A_k} |f_n(z)|.$ Then $\frac 12 < min_n < max_n < 2.$ Choose $r_{n+1}$ and $\delta_{n+1}$ small enough such that 
\[
\ \max_{z\in \cup_{k=0}^n A_k} |b_{n+1}(z)| < \min (min_n - \frac 12, 2-max_n),
\]
$\frac {9}{8} < |b_{n+1}(z)| < \frac {11}{8}$ and $|f_n(z)| < \frac {5}{8}$ (since $|f_n(z_{n+1})| \le \frac 12$) for $z\in A_{n+1},$ and (1) holds.
Then $f_{n+1}(z) = f_n(z) + b_{n+1}(z)$ satisfies (2).

By construction, $K_1$ has no interior and $\{z_n\}_{n=1}^\i$ is dense in $K_1$ since $\{\lambda_k\}$ is dense in $\D(0, \frac 12),$ $0\in K_1,$ and \eqref{K1ExistEq2} follows from \eqref{K1ExistEq1}. (3) is proved. 
\end{proof}

\begin{lemma} \label{KExist} 
Let $\{z_n, r_n, \delta_n,f_n\}_{n=1}^\i,$ $\{A_k\}_{k = 0}^\i,$ and $K_1$ be as in Lemma \ref{K1Exist}.  Set 
\begin{eqnarray} \label{KExistEq1}
\ K = K_1 \cup \bigcup_{n=0}^\i A_n.
\end{eqnarray}
Then $K$ is a compact subset and the following properties hold.
\newline
(1) The envelop $E(K, \area_K)$ satisfies $K = \overline{E(K, \area_K)}$ and $\area(K\setminus E(K, \area_K)) = 0.$
\newline
(2) $0\in E(K, \area_K).$ Therefore, $\phi_0(b) = \rho_{K, \area_K}(b)(0)$ for $b\in R^\i(K, \area_K)$ is a weak-star continuous homomorphism on $R^\i(K, \area_K).$
\newline
(3) There exists $f\in R^\i(K, \area_K)$ such that $\|f\| \le 2,$ $|f(z)| \ge \frac 12$ for $z\in \cup_{n=0}^\i A_n,$ and $\phi_0(f) = 0.$ Hence, $f$ is not invertible in $R^\i(K, \area_K).$
\end{lemma}

\begin{proof}
Let $\mu$ be the arc length measure of $\partial_e K = \partial \D \cup \cup _{n=1}^\i \partial \D(z_n, \delta_n).$ Then $\mu \in M_0^+(K)$ since by \eqref{K1ExistEq1}, we have
\[
\ \|\mu\| \le 2\pi \left (1 + \sum_{n=1}^\i\delta_n \right ) \le 4 \pi.
\]
Let $g = \frac{1}{2\pi i}\frac{dz}{d\mu}.$ Then $|g(z)| =  \frac{1}{2\pi}~\mu-a.a.$ and $\int bgd\mu = 0$ for $b \in R^\i (K,\mu).$ Thus, $R^\i (K,\mu)$ is pure.
For $\lambda \in K \setminus \partial_e K$ with $\widetilde {g\mu}(\lambda) < \i$ and $h\in \text{Rat}(K),$ we have 
\begin{eqnarray}\label{Eq00}
\ h(\lambda) = \int_{\partial \D \cup \cup _{n=1}^m \partial \D(z_n, \delta_n)} \dfrac{h(z)g(z)}{z - \lambda} d \mu(z) \rightarrow \int h(z)g_\lambda (z) d\mu(z) \text{ as }m\rightarrow \i, 
\end{eqnarray}
where $g_\lambda (z) = \frac{g(z)}{z - \lambda}.$ Hence, $\lambda\in E(K,\mu),$ which implies $\area (K \setminus E(K,\mu)) = 0$ since $\widetilde {g\mu}(\lambda) < \i~\area-a.a.$ and $\area(K \setminus \partial_e K) = \area(K) > 0.$ Applying Proposition \ref{EProp} (c), we have 
\[
\ \bigcup_{n=0}^\i \text{int}(A_n) \subset E(K,\mu) \subset K. 
\]
Hence, $K = \overline{E(K,\mu)}$ since $\{z_n\}$ is dense in $K_1.$

Using Chaumat's theorem (\cite{cha74} or \cite[Chaumat's Theorem on page 288]{conway}),
the Chaumat's mapping $\rho_{K,\mu}$ is an isometric isomorphism and a weak-star homeomorphism from $\rikmu$ onto $R^\i (\overline {E(K,\mu)}, \area_{E(K,\mu)}) = R^\i (K, \area_K).$ Hence, $E(K,\area_K) = E(K,\mu)$ by Proposition \ref{EProp} (a).
(1) is proved.

Using \eqref{K1ExistEq2}, we have
$\widetilde {g\mu}(0) < \i.$ Hence, $0\in E(K,\area_K).$ (2) follows from Proposition \ref{EProp} (a).

Since $|f_n(z)| < 2$ on $\cup_{n=0}^\i A_n$ (Lemma \ref{K1Exist} (2)) and $K =\overline {\cup_{n=0}^\i A_n},$ we see that $\|f_n\| \le 2.$ We can choose a sequence $\{f_{n_j}\}$ such that $f_{n_j}$ converges to $f\in R^\i (K, \area_K)$ in weak-star topology. 
Clearly, $\|f\| \le 2.$ For $n \ge 1$ and $\lambda \in \text{int} (A_n) \subset E(K,\area_K),$ by Proposition \ref{EProp} (a), $\phi_\lambda (b) = b(\lambda)$ for $b \in \text{Rat}(K)$ extends a weak-star continuous homomorphism. Hence,
\[
\ \phi_\lambda (f_{n_j}) = f_{n_j}(\lambda) \rightarrow \phi_\lambda (f) = f(\lambda) \text{ as }j\rightarrow \i,
\]
which implies $|f(\lambda)| \ge \frac 12$ since $|f_{n_j}(\lambda)| > \frac 12$ by Lemma \ref{K1Exist} (2). By (2), 
\[
\ \phi_0(f) = \lim_{j\rightarrow\i}\phi_0 (f_{n_j}) = \lim_{j\rightarrow\i}f_{n_j}(0) = 0.
\] 
Therefore, $f$ is not invertible in $R^\i(K, \area_K).$ This completes the proof of (3).
\end{proof}

\begin{lemma} \label{ExampleLemma}
Let $\{A_n\}_{n=0}^\i$ and $K$ be as in Lemma \ref{K1Exist} and Lemma \ref{KExist}. Then there exists a continuous function $W(z)$ on $\C$ satisfying $W(z) > 0$ for $z\in \cup_{n=0}^\i \text{int}(A_n)$ and $W(z) = 0$ for $z\in \C \setminus \cup_{n=0}^\i \text{int}(A_n)$ such that the following properties hold.
\newline
(a) The following decomposition hold:
\[
\ R^2(K, W\area) = \bigoplus_{n=0}^\i R^2(A_n, W\area_{A_n}). 
\]
\newline
(b) $R^\i(K, \area_K) \subset R^\i(K, W\area),$ $\area (K \setminus E(K, W\area)) = 0,$ and $0\in E(K, W\area).$ 
\newline
(c) Let $f$ be the function constructed as in Lemma \ref{KExist}. Then $f \in R^\i(K, W\area)$ and $f$ is invertible in $M(S_{W\area}),$ but $f$ is not invertible in $R^\i(K, W\area).$
\end{lemma}

\begin{proof}
Define
\[
\ W(z) = \begin{cases} (r_n - |z - z_n|)^4, & \dfrac{\delta_n + r_n}{2} < |z - z_n| \le r_n\\ (|z - z_n| - \delta_n)^4, & \delta_n  \le |z - z_n| \le \dfrac{\delta_n + r_n}{2}\\ 0, & z\in \C \setminus \cup_{n=0}^\i A_n.  \end{cases}
\]
Then $W(z)$ is continuous on $\C$ such that $W(z) > 0$ for $z\in \cup_{n=0}^\i \text{int}(A_n)$ and $W(z) = 0$ for $z\in \C \setminus \cup_{n=0}^\i \text{int}(A_n).$ Moreover, for $\lambda \notin A_n$ and $z\in A_n,$ we have
\begin{eqnarray}\label{Eq0}
\ W(z) \le |z-\lambda|^4\text{ and }W(z) \le |z-z_n|^4.	
\end{eqnarray}
For 
$\lambda \in K_1$ defined in Lemma \ref{K1Exist} (3) and $g \in R^2(K, W\area)^\perp,$ where
\[
\ R^2(K, W\area)^\perp := \left \{g\in L^2(W\area):~ \int h(z)g(z) Wd \area = 0 \text{ for }h \in R^2(K, W\area) \right \},
\]
we see that, from \eqref{Eq0},
\begin{eqnarray}\label{Eq1}
\ \widetilde {gW\area}(\lambda) \le \left (\int \dfrac{1}{|z-\lambda|^2}Wd\area \right )^{\frac 12} \|g\| \le 2 \sqrt{\area (\D)} \|g\| \le 2\sqrt{\pi} \|g\|.
\end{eqnarray}
Hence, $\mathcal C(gW\area)(\lambda)$ is well defined. Also for $\lambda \in K_1,$ using \eqref{Eq0}, we have
\[
\begin{aligned}
\ & \left | \mathcal C(gW\area)(\lambda) -   \mathcal C(gW\area)(z_n)\right | \\
\ \le &	|z_n - \lambda| \left (\int \dfrac{1}{|z-z_n|^4}Wd\area \right )^{\frac 14} \left (\int \dfrac{1}{|z-\lambda|^4}Wd\area \right )^{\frac 14} \|g\| \\
\ \le &	\pi |z_n - \lambda| \|g\|.
\end{aligned}
\]
Since $\{z_n\}_{n=1}^\i$ is dense in $K_1$ by Lemma \ref{K1Exist} (3), for $\lambda \in K_1,$ we select a subsequence $\{z_{n_k}\}$ such that $z_{n_k}$ tends to $\lambda,$ which implies $\mathcal C(gW\area)(\lambda) = 0$ since $\mathcal C(gW\area)(z_n) = 0.$ Therefore, 
\begin{eqnarray}\label{Eq2}
\ \mathcal C(gW\area)(\lambda) = 0, ~ \lambda \in K_1.
\end{eqnarray}
Thus, using \eqref{Eq1}, Fubini's Theorem, and \eqref{Eq2}, we get for $n \ge 1,$
\[
\begin{aligned}
\ \int \chi_{A_n}g Wd\area = & \dfrac{1}{2\pi i} \int \int_{\partial \D(z_n,r_n)} \dfrac{dz}{z - w} gWd\area(w) \\ 
\ = & - \dfrac{1}{2\pi i}\int_{\partial \D(z_n,r_n)} \mathcal C(gW\area)(z)dz \\
\ = & 0.
\end{aligned}
\]
Thus, $\chi_{A_n}\in R^2(K, W\area)$ for $n \ge 1.$ Thus, $\chi_{A_0}\in R^2(K, W\area)$ since $\chi_{A_0} = 1 - \sum_{n=1}^\i \chi_{A_n},~W\area-a.a..$ It is clear that $R^2(K, W\area_{A_n}) = R^2(A_n, W\area_{A_n}).$ (a) is proved.

For (b): we have, 
\[
\ R^\i(K, \area_K) \subset R^\i(K, W\area) \text{ and } \cup_{n=0}^\i \text{int}(A_n) \subset E(K, W\area). 
\]
Let $T=\partial\D(0,\frac{3}{4})\cup \cup_{n=1}^\i \partial \D(z_n,\frac{\delta_n+r_n}{2}).$ Using the same argument as in \eqref{Eq00},  for $\lambda \in K_1$ with $\int_T \frac{1}{|z - \lambda|} |dz| < \i,$ we conclude that 
\[
\ \phi_\lambda (b) = \dfrac{1}{2\pi i}\int_T \dfrac{b(z)}{z - \lambda} d z = b(\lambda) \text{ for } b\in\text{Rat}(K). 
\]
For $h\in R^\i(K, W\area),$ let $\{h_n\}\subset \text{Rat}(K)$  such that $h_n\rightarrow h$ in $L^\i(W\area)$ weak-star topology. Then there exists $C_1 > 0$ such that
$\|h_n\|_{L^\i(W\area)} \le C_1.$ Clearly, $h_n\rightarrow h$ in $L^\i(\area_{A_n})$ weak-star topology. Hence,
$h_n(z)\rightarrow h(z)$ uniformly on any compact subset of $\text{int}(A_n)$ and $h$ is analytic on $\cup_{n=0}^\i \text{int}(A_n).$ Thus, $h_n(z)\rightarrow h(z),~z\in T,~ |dz|-a.a..$ Using the Lebesgue dominated convergence theorem, we have
\[
\ \dfrac{1}{2\pi i}\int_T \dfrac{h_n(z)}{z - \lambda} d z = h_n(\lambda) \rightarrow \dfrac{1}{2\pi i}\int_T \dfrac{h(z)}{z - \lambda} d z.
\]
Therefore, $\phi_\lambda$ extends a weak-star continuous homomorphism on $R^\i(K, W\area),$ which implies $\lambda \in E(K, W\area)$ by Proposition \ref{EProp} (a). Thus, $\area (K_1 \setminus E(K, W\area)) = 0$ since  $\int_T \frac{1}{|z - \lambda|} |dz| < \i~\area_{K_1}-a.a..$ Therefore, $\area (K \setminus E(K, W\area)) = 0$ since $\text{int}(A_n) \subset E(K, W\area)$ for all $n \ge 0.$ The fact $0 \in E(K, W\area)$ follows from \eqref{K1ExistEq2}.

(c): Clearly,
\[
\ M(S_{W\area_{A_n}}) = R^\i (A_n, \area_{A_n}) = H^\i (\text{int}(A_n)).
\]
Since $|f(z)| \ge \frac 12, ~ z\in \cup_{n=0}^\i A_n,$ the function $f|_{A_n}$ is invertible in $M(S_{W\area_{A_n}})$ and $\|(f|_{A_n})^{-1}\| \le 2.$ Define
$h = \sum _{n=0}^\i (f|_{A_n})^{-1}.$
Then $\|h\| \le 2,$ $h \in M(S_{W\area})$ by (a), and $fh = 1,~W\area-a.a..$ Therefore, 
 we conclude that $f$ is invertible in $M(S_{W\area})$ and $f$ is not invertible in $R^\i(K, W\area)$ due to $0 \in E(K, W\area)$ and Lemma \ref{KExist}. This proves (c).
\end{proof}

The proof of Theorem \ref{ExampleThm} follows from Lemma \ref{ExampleLemma}.

\section{\textbf{Invertibility in M(S)}}

We assume that $S\in \mathcal L(\mathcal H)$ is a subnormal operator and $\mu$ is the scalar-valued spectral measure for $N=mne(S)$ on $\mathcal K,$ where $\mathcal K$ is defined as in \eqref{NEHSpace}.

For a Borel subset $\Delta,$ denote $f_{\Delta} = f \chi_{\Delta}$ for $f\in L^\infty (\mu).$
Let $\{H_k\}\subset \mathcal H$ and $\{G_j\}\subset \mathcal K \ominus \mathcal H$ be norm dense subsets of $\mathcal H$ and $\mathcal K \ominus \mathcal H,$ respectively. Let
\[
\ \Delta_1 = \bigcup_{k,j=1}^\infty \{z:~ \langle H_k(z),G_j(z)\rangle \ne 0 \} \text{ and } \Delta_0 = \text{spt} \mu \setminus \Delta_1.  
\]
Let $\mathcal K_{\Delta_0} = \{f_{\Delta_0}:~ f\in \mathcal K\}$ and $N_{\Delta_0} = M_z |_{\mathcal K_{\Delta_0}}.$ Then $\mathcal K_{\Delta_0} \subset \mathcal H,$ $S\mathcal K_{\Delta_0} \subset \mathcal K_{\Delta_0},$ and $N_{\Delta_0} = S |_{\mathcal K_{\Delta_0}}.$ Set $\mathcal H_{\Delta_1} = \{f_{\Delta_1}:~ f\in \mathcal H\}.$ Then $\mathcal H_{\Delta_1} \subset \mathcal H,$ $ \mathcal K_{\Delta_0} \perp \mathcal H_{\Delta_1},$ and $ \mathcal  H = \mathcal K_{\Delta_0} \oplus \mathcal H_{\Delta_1}.$
Define $S_{\Delta_1} = S |_{\mathcal H_{\Delta_1}}.$ Clearly, we have the following decomposition:
\begin{eqnarray} \label{MDecompEq}
\ S = N_{\Delta_0} \oplus S_{\Delta_1} \text{ and } M(S) = L^\infty (\mu_{\Delta_0}) \oplus M(S_{\Delta_1}).
\end{eqnarray}
Therefore, we get the following proposition. 

\begin{proposition}\label{SMDepProp}
Let $\Delta_0,$ $\Delta_1,$ and $S_{\Delta_1}$ be defined as above. Then $S_{\Delta_1}$ is a pure subnormal operator and $M(S_{\Delta_1})$ contains no non-trivial direct $L^\infty$ summands. Consequently, $S$ is pure if and only if $M(S)$ contains no non-trivial direct $L^\infty$ summands.	
\end{proposition}
\cite[Theorem III.1]{dud84} provides a similar decomposition of $S$ for $R^\i (\sigma(S), \mu).$

Due to the spectral behavior of direct sums and the spectral mapping theory developed for normal operators, \eqref{MDecompEq} shows that for our purposes no real loss of generality results from restricting our attention to pure subnormal operators.

In the remaining section, we assume that $S$ is a pure subnormal operator. For $H\in \mathcal H$ and $G\in \mathcal K\ominus \mathcal H,$ let $\mu_{H,G} = \langle H(z), G(z)\rangle \mu$ and define
\[
\ E_{H,G} = \left \{\lambda\in \sigma(S):~ \widetilde {\mu_{H,G}}(\lambda) < \i \text{ and }\mathcal C(\mu_{H,G})(\lambda) \ne 0\right \}
\]
(notice that $\tilde \nu$ is defined as in \eqref{CTDefInit2}).
The envelope for $S$ is defined
\begin{eqnarray} \label{ESDef}
\ E_S: = \bigcup_{H\in \mathcal H,~G\in \mathcal K\ominus \mathcal H} E_{H,G}.
\end{eqnarray}

\begin{example} \label{ESExample} 
Let $\{A_n\}_{n=0}^\i,$ $K,$ and $W(z)$ be as in Lemma \ref{K1Exist}, Lemma \ref{KExist}, and Lemma \ref{ExampleLemma}, respectively. Then
\[
\ E_{S_{W\area}} \approx \cup_{n=1}^\i \text{int}(A_n) \text{ and } E_{S_{W\area}}^\i  \approx K,~ \area-a.a..
\]
Consequently, $\area (E_{S_{W\area}}^\i \setminus E_{S_{W\area}}) > 0.$
\end{example}

For $C_1,C_2 > 0,$ the statement $C_1 \lesssim C_2$ means: there exists an absolute constant $C>0,$ independent of $C_1$ and $C_2,$ such that $C_1\le CC_2.$

\begin{proposition}\label{ESProp} The following properties hold.
\newline
(a) $E_S \subset E_S^\i.$
\newline
(b) There exist norm dense subsets $\{H_k\}\subset \mathcal H$ and $\{G_j\}\subset \mathcal K \ominus \mathcal H$ and there ecists a
doubly indexed  sequence of open balls $\{\D (\lambda_{nm}, \delta_{nm})\}$ such that
\[
\ E_S \approx  \bigcup_{k,j=1}^\infty E_{H_k,G_j} \approx \bigcap_{m=1}^\infty \bigcup_{n=1}^\infty \D (\lambda_{nm}, \delta_{nm}), ~ \area-a.a.
\]
Consequently, $E_S$ is $\area$ measurable subset.
\newline
(c) $E_S$ has area density one at each of its points.
\end{proposition}

\begin{proof} (a): Trivial.

For $H\in \mathcal H,$ $G\in \mathcal K\ominus \mathcal H,$ and $\lambda\in E_{H, G},$ we have
\[
\ E_{H, G}^c \subset \{|\CT(\mu_{H, G})(z) - \CT(\mu_{H, G})(\lambda)| > |\CT(\mu_{H, G})(\lambda)|/2\},~\area-a.a..
\]
Hence,
\begin{eqnarray}\label{ESPropEq}	
\begin{aligned}
\  & \dfrac{\area (\D(\lambda, \delta) \setminus  E_{H, G})}{\delta ^2} \\
\ \lesssim & \dfrac{2}{|\CT(\mu_{H, G})(\lambda)|\delta^2} \int_{\D(\lambda,\delta)} |\CT(\mu_{H, G})(z) - \CT(\mu_{H, G})(\lambda)| d \area (z)\\
\ \lesssim & \dfrac{1}{\delta^2} \left (\int_{\D(\lambda,\sqrt{\delta})^c} + \int_{\D(\lambda,\sqrt{\delta})} \right )\int_{\D(\lambda,\delta)} \dfrac{|z-\lambda|d\area(z)}{|w-\lambda||w-z|}d|\mu_{H, G}|(w)\\
\ \lesssim & \dfrac{1}{\delta^2} \left (\dfrac{\pi \delta^3}{\sqrt{\delta} - \delta}  \widetilde{\mu_{H, G}} (\lambda) + \delta^2\int_{\D(\lambda,\sqrt{\delta})} \dfrac{1}{|w-\lambda|}d|\mu_{H, G}|(w) \right ) \\
\ & \rightarrow 0 \text{ as } \delta \rightarrow 0.
\end{aligned}
\end{eqnarray} 
Using the same proof of \cite[Proposition 5.8]{y23}, we get (b). 

(c) follows from \eqref{ESPropEq}	.
\end{proof}

In this section, we prove the following theorem.

\begin{theorem}\label{MTheorem}
Suppose that $S$ is pure. Then for $f\in R^\i(\sigma (S), \mu),$ $f$ is invertible in $M(S)$ if and only if there exists $\epsilon_f > 0$ such that 
\[
\ |\rho_S(f)(z)| \ge \epsilon_f,~ \area_{E_S}-a.a..
\]
\end{theorem}

 \subsection{Extension of the Chaumat's map $\rho_S$}
 For a measurable subset $B\subset \C,$ let $L_0^\i(\C,B)$ be the space of bounded measurable functions $\varphi,$ where $\varphi$ is compactly supported and $\varphi (z) = 0 ~\area_B-a.a..$ For $H\in \mathcal H,$ $G\in \mathcal K \ominus \mathcal H,$ and $\varphi \in L_0^\i(\C,E_S),$ we see that $\CT(\varphi\area)$ is continuous on $\C$ and by the Fubini's theorem,
 \[
 \ \int \CT(\varphi\area) d\mu_{H,G} = - \int \varphi \CT \mu_{H,G} d\area = 0.
 \]
 Thus, $\CT(\varphi\area) \in M(S).$ Moreover, for $\widetilde{\mu_{H,G}}(\lambda) < \i$ and $\varphi_\delta = \varphi \chi_{\D(\lambda, \delta)^c}$ we have
 \[
 \ \int \dfrac{\CT(\varphi_\delta\area)(z) - \CT(\varphi_\delta\area)(\lambda)}{z - \lambda} d\mu_{H,G} (z) = - \int \mathcal C(\mu_{H,G}) (w) \varphi_\delta (w)\dfrac{d\area (w)}{w - \lambda} = 0.
 \]
 Hence,
 \[
 \ \CT(\varphi_\delta\area)(\lambda)\mathcal C(\mu_{H,G}) (\lambda) = \mathcal C(\CT(\varphi_\delta\area)\mu_{H,G}) (\lambda).
 \]
 Since $\CT(\varphi_\delta\area) \rightarrow \CT(\varphi\area)$ uniformly as $\delta\rightarrow 0,$ we get
 \begin{eqnarray}\label{KeyIdentity}
 \ \CT(\varphi\area)(\lambda)\mathcal C(\mu_{H,G}) (\lambda) = \mathcal C(\CT(\varphi\area)\mu_{H,G}) (\lambda)
 \end{eqnarray}
 for $H\in \mathcal H,$ $G\in \mathcal K \ominus \mathcal H,$ $\widetilde{\mu_{H,G}}(\lambda) < \i,$ and $\varphi \in L_0^\i(\C,E_S).$

  Let $R(\overline E_S, E_S)$ be the uniform closed sub-algebra of $C(\overline E_S)$ generated by $\CT(\varphi\area),$ where $\varphi \in L_0^\i(\C, E_S).$ Clearly, $\text{Rat}(\overline E_S) \subset R(\overline E_S, E_S).$ Let $M_a(S)$ be the $L^\i(\mu)$ weak-star closure of $R(\overline E_S, E_S).$ Then
  \[
  \ R^\i(\sigma(S),\mu) \subset M_a(S) \subset M(S).
  \]
  Let $R^\i(\overline E_S, E_S)$ be the $L^\i(\area_{E_S})$ weak-star closure of $R(\overline E_S, E_S).$
  
 In Example \ref{ESExample}, we see that 
\[
\ M_a(S_{W\area}) = M(S_{W\area}) = R^\i(\overline E_{W\area}, E_{W\area}) = \oplus_{n=1}^\i H^\i(\text{int}(A_n)) \supsetneqq R^\i(K, W\area) .
\]

\begin{proposition}\label{MAProp}
For $H\in \mathcal H$ and $G\in \mathcal K \ominus \mathcal H,$ the following properties hold.
\newline
(1) If $\lambda \in  E_S$ such that $\widetilde{\mu_{H,G}}(\lambda) < \i,$ then
\begin{eqnarray}\label{MAPropEq0}	
\ f(\lambda)\mathcal C(\mu_{H,G}) (\lambda) = \mathcal C(f\mu_{H,G}) (\lambda) ~ \text{ for } f\in R(\overline E_S, E_S).
\end{eqnarray}
\newline
(2) If $\lambda \in  E_S$ such that $\widetilde{\mu_{H,G}}(\lambda) < \i$ and $\CT (\mu_{H,G})(\lambda) \ne 0,$  then $\phi_\lambda (f) = \frac{\CT (f\mu_{H,G})(\lambda)}{\CT (\mu_{H,G})(\lambda)}$ is a weak-star continuous homomorphism on $M_a(S)$ and
$\rho_S(f)(\lambda) = \phi_\lambda (f)$ for $f\in R^\i(\sigma(S),\mu).$
Consequently, we extend the definition of the Chaumat's map by setting $\rho_S(f)(\lambda) = \phi_\lambda (f)$ for $\lambda\in E_S$ and $f\in M_a(S).$
\newline
(3) If $f_1,f_2\in M_a(S),$ then $\rho_S(f_1f_2)(z) = \rho_S(f_1)(z)\rho_S(f_2)(z)~\area-a.a..$
\newline
(4) If $f\in M_a(S),$ then $\|\rho_S(f)\|_{L^\i (\area_{E_S})} \le \|f\|_{L^\i (\mu)}.$
\newline
(5) If $f\in M_a(S),$ then $\rho_S(f) \in R^\i(\overline E_S, E_S)$ and
\[
\ \rho_S(f) (z) \CT (\mu_{H,G})(z) = \CT (f\mu_{H,G})(z)~ \area-a.a..
\]
\end{proposition}

\begin{proof} (1):
For $\varphi_1,\varphi_2\in L_0^\i(\C,E_S),$ we see that $\CT(\varphi_2\area)H\in \mathcal H$ and $\widetilde{\mu_{\CT(\varphi_2\area)H,G}}(\lambda) = \widetilde{\CT(\varphi_2\area)\mu_{H,G}}(\lambda) < \i.$ Using \eqref{KeyIdentity}, we get 
\[
\begin{aligned}
 \ \CT(\varphi_1\area)(\lambda)\CT(\varphi_2\area)(\lambda)\mathcal C(\mu_{H,G}) (\lambda) = & \CT(\varphi_1\area)(\lambda)\mathcal C(\mu_{\CT(\varphi_2\area)H,G}) (\lambda) \\
 \  = & \mathcal C(\CT(\varphi_1\area)\CT(\varphi_2\area)\mu_{H,G}) (\lambda).
 \end{aligned}
 \]
 This proves \eqref{MAPropEq0}.
 
(2): By definition, $\phi_\lambda$ is a weak-star continuous. For $f, f_1,f_2\in R(\overline E_S, E_S),$ using \eqref{MAPropEq0}, we have 
\[
\ \phi_\lambda(f) = f(\lambda)\text{ and }\phi_\lambda(f_1f_2) = (f_1f_2)(\lambda) = \phi_\lambda(f_1)\phi_\lambda(f_2).  
\] 
Thus, $\phi_\lambda$ is homomorphism. For $r \in \text{Rat}(\sigma(S)),$ we have $\phi_\lambda(r) = r(\lambda) = \rho_S(r).$ So $\rho_S(f)(\lambda) = \phi_\lambda (f)$ for $f\in R^\i(\sigma(S),\mu).$

(3): Trivial from (2).

(4): For $f\in M_a(S)$ and $\lambda \in  E_S$ such that $\widetilde{\mu_{H,G}}(\lambda) < \i$ and $\CT (\mu_{H,G})(\lambda) \ne 0,$  from (2) , we have
\[
\ |(\rho_S(f)(\lambda))^n| = \left | \int f^n \dfrac{d\mu_{H,G}/(z-\lambda)}{\CT (\mu_{H,G})(\lambda)} \right |\le \|f\|_{L^\i (\mu)}^n \dfrac{\widetilde{\mu_{H,G}}(\lambda)}{|\CT (\mu_{H,G})(\lambda)|}.  
\]
Hence,
\[
\ |\rho_S(f)(\lambda)| \le \|f\|_{L^\i (\mu)}\left ( \dfrac{\widetilde{\mu_{H,G}}(\lambda)}{|\CT (\mu_{H,G})(\lambda)|} \right )^{\frac 1n} \rightarrow \|f\|_{L^\i (\mu)} \text{ as }n \rightarrow \i.  
\]
(4) is proved.

(5) Let $f_n\in R(\overline E_S, E_S)$ such that $f_n\rightarrow f\in M_a(S)$ in $L^\i(\mu)$ weak-star topology. From (4) and (2), we get $\sup_n \|f_n\|_{L^\i(\area_{E_S})} \le \sup_n \|f_n\|_{L^\i(\mu)} < \i$ and $f_n(\lambda) \rightarrow \rho_S(f)(\lambda)$ for $\lambda\in E_S.$ Hence, $\rho_S(f_n) \rightarrow \rho_S(f)$ in $L^\i(\area_{E_S})$ in weak-star topology, which implies $\rho_S(f) \in R^\i(\overline E_S, E_S).$ (5) follows.
\end{proof}

\begin{theorem}\label{MAEqualREE}
 The Chaumat's map $\rho_S$ extends an isometric isomorphism and a weak-star homeomorphism  from $M_a(S)$ onto $R^\i(\overline E_S, E_S).$ 
 \end{theorem}
 
 \begin{proof}
 It is clear that, from Proposition \ref{MAProp} (5), the map $\rho_S$ is surjective and injective since $S$ is pure. By Proposition \ref{MAProp} (4), $\rho_S$ is bounded linear operator from $M_a(S)$ onto $R^\i(\overline E_S, E_S).$ 
 Using the bounded inverse theorem, we infer that $\rho_S^{-1}$ is bounded.
 Therefore, $\rho_S$ is bijective      
 isomorphism between two Banach algebras $M_a(S)$ onto $R^\i(\overline E_S, E_S).$ Clearly, $\rho_S$ is also a weak-star sequentially continuous, so an application of Krein-Smulian Theorem shows that $\rho_S$ is a weak-star homeomorphism.	
 \end{proof}

\subsection{Invertibility in $M_a(S)$ and Proof of Theorem \ref{MTheorem}} 
We start with the following proposition.

\begin{proposition}\label{OpInvertibleForS}
Let $S\in \mathcal L(\mathcal H)$ be a pure subnormal operator and let $\mu$ be the scalar-valued spectral measure for $N=mne(S).$
If $f\in M_a(S)$ and $M_{S,f}$ is invertible, then  there exists $\epsilon_f > 0$ such that
 \begin{eqnarray}\label{OpInvertibleForSEq1} 
 \ |\rho_S(f)(z) | \ge \epsilon_f, ~\area_{E_S}-a.a.. 
 \end{eqnarray}
 \end{proposition}

\begin{proof} 
Let $T\in \mathcal L (\mathcal H)$ such that $M_{S,f}T = I.$ Then for $H\in \mathcal H$ and $G\in \mathcal K \ominus \mathcal H,$ from Proposition \ref{MAProp} (5), we have
 \[
 \ (\rho_S(f)(z))^n\mathcal C(\mu_{T^nH, G})(z) = \mathcal C(\mu_{H, G})(z), ~ \area_{E_S}-a.a..
 \]
 Therefore, 
 \[
 \begin{aligned}
 \ & \left (\int_{\{\mathcal C(\mu_{H, G})(z) \ne 0\}}\dfrac{1}{|\rho_S(f)(z)|^n}\mathcal C(\mu_{H, G})(z)| d \area (z)\right )^\frac{1}{n} \\
 \  \le & \left (C_1\sqrt{\area(E_S)}\|H\| \|G\|\right )^\frac{1}{n} \|T\|,
 \end{aligned}
 \] 
where $C_1$ is an absolute constant. Taking $n\rightarrow \i$ and using Proposition \ref{ESProp} (b), we prove \eqref{OpInvertibleForSEq1}. 
\end{proof}

For a bounded Borel subset $\mathcal D\subset \mathbb C,$ let 
  $H(\mathcal D)$ denote the set of functions with analytic extensions to open neighborhoods of $\mathcal D$ and let $H^\i (\mathcal D)$ be the weak-star closure in $L^\i(\area_{\mathcal D})$ of $H(\mathcal D).$
  
  If $B \subset\C$ is a compact subset, then  we define the analytic capacity of $B$ by
\[
\ \gamma(B) = \sup |f'(\infty)|,
\]
where the supremum is taken over all those functions $f$ that are analytic in $\mathbb C_{\infty} \setminus B$ ($\mathbb{C}_\infty := \mathbb{C} \cup \{\infty \}$) such that
$|f(z)| \le 1$ for all $z \in \mathbb{C}_\infty \setminus B$; and
$f'(\infty) := \lim _{z \rightarrow \infty} z(f(z) - f(\infty)).$
The analytic capacity of a general subset $E$ of $\mathbb{C}$ is given by: 
$\gamma (E) = \sup \{\gamma (B) : B\subset E \text{ compact}\}.$ 
We use $\gamma-a.a.$ for a property that holds everywhere except possibly on a set of analytic capacity zero.
The following elementary property can be found in \cite[Theorem VIII.2.3]{gamelin},
\begin{eqnarray}\label{AreaGammaEq}
\ \area(E) \le 4\pi \gamma(E)^2,
\end{eqnarray}
 
 X. Tolsa has established the following astounding theorems $\gamma$ (see \cite{Tol03} or \cite[Theorem 6.1]{Tol14}): There exists an absolute constant $C_T > 0$ such that for $E \subset \mathbb{C},$
 \begin{eqnarray} \label{GammaEqualGammaP}
 \ \gamma (E) \le C_T \sup \{\|\eta\|:~ \eta \in M_0^+(E) \text{ and }\|\CT\eta\|_{L^\i(\C)} \le 1\}	
 \end{eqnarray}

 and for $E_1,E_2,...,E_m \subset \mathbb{C}$ ($m$ may be $\i$),
\begin{eqnarray} \label{SAAC}
\ \gamma \left (\bigcup_{i = 1}^m E_i \right ) \le C_T \sum_{i=1}^m \gamma(E_i).
\end{eqnarray}

The following theorem implies Theorem \ref{MTheorem} since $R^\i (\sigma (S), \mu) \subset M_a(S).$ 
 
\begin{theorem} \label{InvInMA}
Let $S\in \mathcal L(\mathcal H)$ be a pure subnormal operator and let $\mu$ be the scalar-valued spectral measure for $N=mne(S).$ Then for $f\in M_a(S),$ the following statements are equivalent:
\newline 
(1) $\rho_S(f)$ is invertible in $R^\i(\overline E_S, E_S).$
\newline 
(2) $f$ is invertible in $M_a(S).$
\newline 
(3) $f$ is invertible in $M(S).$
\newline 
(4) There exists $\epsilon_f > 0$ such that 
\[
\ |\rho_S(f)(z)| \ge \epsilon_f,~ \area_{E_S}-a.a..
\] 
\end{theorem}

\begin{proof}
(1)$\Rightarrow (2)$ follows from Theorem \ref{MAEqualREE}. 

(2)$\Rightarrow (3)$ is trivial.	

(3)$\Rightarrow (4)$ follows from Proposition \ref{OpInvertibleForS}.

(4)$\Rightarrow (1):$ Let $E$ be the envelope for $R^\i(\overline E_S, E_S),$ the set of points $\lambda \in \C$ such that there exists $g\in R^\i(\overline E_S, E_S)^\perp$ satisfying $\widetilde{g\area_{E_S}}(\lambda) < \i$ and $\CT(g\area_{E_S})(\lambda) \ne 0.$

As shown in \cite[Proposition 5.7]{y23}, we conclude that there exists a sequence $\{g_n\} \subset R^\i(\overline E_S,E_S)^\perp$ and there exists a
doubly indexed  sequence of open balls $\{\D (\lambda_{ij}, \delta_{ij})\}$ such that
\begin{eqnarray}\label{EEqn}
\ E \approx  \bigcup_{n=1}^\infty E_0(g_n\area_E) \approx \bigcap_{j=1}^\infty \bigcup_{i=1}^\infty \D (\lambda_{ij}, \delta_{ij}), ~ \area-a.a.,
\end{eqnarray}
where $E_0(g_n\area_E) = \{\widetilde{g_n\area_E}(z) < \i\text{ and }\CT(g_n\area_E)(z) \ne 0\}.$ We may assume that $\{g_n\}$ is $L^1(\area_{E_S})$ norm dense in  $R^\i (\overline E_S,E_S)^\perp.$
Thus, Proposition \ref{ESProp} holds for $E.$

Using Proposition \ref{MAProp} (2) and Theorem \ref{MAEqualREE}, we infer that for $\lambda\in E_S,$ the evaluation $\psi_\lambda (f) = f(\lambda)$ for $f\in R(\overline E_S, E_S)$ extends a weak-star continuous homomorphism on $R^\i(\overline E_S, E_S).$ Hence, $E_S \subset E.$ Clearly, if $g\in R^\i(\overline E_S, E_S)^\perp,$ then $\CT(g\area_{E_S})(z) = 0 ~\area_{\C\setminus E_S}-a.a.,$ which implies $\area(E\setminus E_S)=0$ and $\overline E_S = \overline E$ by Proposition \ref{ESProp} (c). Thus, 
\[
\ R^\i(\overline E_S, E_S) = R^\i(\overline E, E).
\]

Now the proof of \cite[Theorem 1.1]{y23} will apply with slight modifications below.
\newline
(1) We replace $R^\i(\overline{E}, \area_E)$ in \cite{y23} by $R^\i(\overline{E}, E).$  Define the removable set
\begin{eqnarray}\label{RSDef}
\ \mathcal R_S = \bigcup_{n = 1}^\i \{z:~\lim_{\epsilon \rightarrow 0} \mathcal C_\epsilon(g_n\area_E)(z)\text{ exists, } \lim_{\epsilon \rightarrow 0} \mathcal C_\epsilon(g_n\area_E)(z) \ne 0\},
\end{eqnarray}
where $\mathcal C_\epsilon(\nu)(z)= \int_{|w-z| >\epsilon} \frac{1}{w-z}d\nu(w)$ for $\nu\in M_0(\C).$
Using the Fubini's theorem and \eqref{EEqn}, we get $E\approx \mathcal R_S~\area-a.a..$
\newline
(2) The proofs of \cite[Theorem 6.1]{y23} (2)$\Rightarrow$(3) and (3)$\Rightarrow$(1) are valid for $R^\i(\overline E,E).$ Clearly, if $g\in H^\i (\mathcal R_S)^\perp,$ then $\CT(g\area_E)(z) = 0 ~\area_E = \area_{\mathcal R_S}-a.a.,$ which implies $R^\i(\overline E,E) \subset H^\i (\mathcal R_S).$  Therefore,
\begin{eqnarray}\label{REEEqualsHIR}
\ R^\i(\overline E,E) = H^\i (\mathcal R_S).
\end{eqnarray}

Slight modifications of \cite[Theorem 6.1]{y23} (1)$\Rightarrow$(2) are needed: Using \eqref{REEEqualsHIR}, we assume that $f_n$ is bounded and analytic in an open neighborhood of $\mathcal R_S$ (replacing $f_n \in R(\overline E)$). Using the same proof \cite[Theorem 6.1]{y23} (3)$\Rightarrow$(1) for $f.$ We see that
\begin{eqnarray}\label{IntEstimateEq}
\ \left | \int f_n(z) \dfrac{\partial \varphi (z)}{\partial  \bar z} d\area_E(z) \right | \lesssim \delta \left \|\dfrac{\partial \varphi (z)}{\partial  \bar z} \right \|_\infty \gamma (\D(\lambda, \delta) \setminus \mathcal R_S). 
\end{eqnarray}
Taking $n\rightarrow \i$ and replacing $\varphi$ by $(z-\lambda)^k\varphi,$ we get \cite[Theorem 6.1]{y23} (2).
 \newline
 (3) In \cite[Lemma 7.2]{y23}, replace $R(K_{\epsilon, N})$ by $H(\mathcal R _{\epsilon, N}),$ where $\mathcal R _{\epsilon, N} =  \mathcal R  \setminus (A_\epsilon\cup \mathcal E_N).$ In its proof, replace $\gamma(\D(\lambda, \delta)\setminus K_{\epsilon, N})$ by $\gamma(\D(\lambda, \delta)\setminus \mathcal R_{\epsilon, N})$ (similar estimate as in \eqref{IntEstimateEq}).
 \newline
 (4) In \cite[Lemma 7.3]{y23} and its proof, replace $\text{Rat}(\overline E)$ by $H(\mathcal R_S).$
 \newline
 (5) In the proof of \cite[Theorem 1.1]{y23}, replace $R(K_{\epsilon, N})$ by $H(\mathcal R _{\epsilon, N}).$
\end{proof}

\section{\textbf{Invertibility of Multiplication Operators}}

We continue to use the assumptions and notation in last section.
The Fredholm index of $T \in \mathcal L(\mathcal H),$ where $0 \notin \sigma_e(T),$ is defined
\[
\ ind(T) = \dim(ker(T)) - \dim(ker(T^*)). 
\]
The objective of this section is to prove the following theorem.

\begin{theorem} \label{theoremA}
Let $S\in \mathcal L(\mathcal H)$ be a pure subnormal operator and let $\mu$ be the scalar-valued spectral measure for $N=mne(S).$ Then for $f\in M_a(S),$ the following statements hold: 
\newline
(1) $M_{S,f}$ is invertible if and only if $f$ is invertible in $M_a(S)$ if and only if $f$ is invertible in $M(S).$
\newline
(2) $M_{S,f}$ is Fredholm if and only if $f = p f_0$, where $p$ is a polynomial with only zeros $z_1,z_2,...,z_m\in \sigma(S)\setminus \sigma_e(S)$, multiplicity of $z_i$ is $n_i$,  $f_0\in M_a(S),$ and $f_0$ is invertible in $M_a(S).$
In this case,
 \[
 \ \text{ind}(M_{S,f}) = \sum_{i=1}^m n_i \text{ind}(S-z_i).
 \] 
\end{theorem} 

As an application to $R^\i(\sigma(S),\mu),$ we get the following corollary.

\begin{corollary} \label{corollaryA}
Let $S\in \mathcal L(\mathcal H)$ be a subnormal operator and let $\mu$ be the scalar-valued spectral measure for $N=mne(S).$ Then for $f \in R^\i(\sigma(S),\mu),$ the following statements hold.
\newline
(1) $M_{S,f}$ is invertible if and only if $f$ is invertible in $M(S).$
\newline
(2) $M_{S,f}$ is Fredholm if and only if $f = p f_0$, where $p$ is a polynomial with only zeros $z_1,z_2,...,z_m\in \sigma(S)\setminus \sigma_e(S)$, multiplicity of $z_i$ is $n_i$, $f_0\in R^\i(\sigma(S),\mu),$ and $f_0$ is invertible in $M(S).$
In this case,
 \[
 \ \text{ind}(M_{S,f}) = \sum_{i=1}^m n_i \text{ind}(S-z_i).
 \] 
\end{corollary}

\begin{proof}
Let $\Delta_0$ and $\Delta_1$ be as in \eqref{MDecompEq}. Then
\[
\ M_{S,f} = M_{N_{\Delta_0},f} \oplus M_{S_{\Delta_1},f}. 
\]
Since $N_{\Delta_0}$ is normal, 	$M_{N_{\Delta_0},f}$ is invertible or Fredholm if and only if $f$ is invertible in $L^\i (\mu_{\Delta_0}).$ Since $S_{\Delta_1}$ is pure by Proposition \ref{SMDepProp}, applying Theorem \ref{theoremA}, we see that $M_{S_{\Delta_1},f}$ is invertible if and only if $f$ is invertible in $M(S_{\Delta_1}).$ Now the corollary follows Theorem \ref{theoremA}.
\end{proof}

\begin{proof} (Theorem \ref{theoremA} (1)):
For $f\in M_a(S),$ if $M_{S,f}$ is invertible in $\mathcal L(\mathcal H),$ then
the invertibility of $f$ in $M_a(S)$ follows from Proposition \ref{OpInvertibleForS} and Theorem \ref{InvInMA}.
\end{proof}

To prove Theorem \ref{theoremA} (2), we need several lemmas. The following lemma is from \cite{cd78}.

\begin{lemma}\label{CowenDouglas}
Suppose that $T\in \mathcal L(\mathcal H),$ $\lambda_0\in \sigma(T) \setminus \sigma_e(T),$ and $\lambda_0$ is not an eigenvalue of $T,$  then there exists a constant $\delta > 0$ and there exist conjugate analytic vector-valued functions $k_\lambda^j\in \mathcal H$ on $\D(\lambda_0, \delta)$ (i.e. $(f, k_\lambda^j)$ is analytic for $f\in \mathcal H$) which are linear independent for $\lambda \in \D(\lambda_0, \delta),~ j=1,2,...n,$ where $n=-ind(T-\lambda_0),$ so that
\[
\ \D(\lambda_0, \delta) \subset  \sigma(T) \setminus \sigma_e(T)
\]
and
\[
\ (T-\lambda )^*k_\lambda^j=0,\text{ for } j = 1, 2, ...n, \text{ and } \lambda \in \D(\lambda_0, \delta).
\]
In addition, if $T$ has no eigenvalues and $\Omega$ is a component of $\sigma(T) \setminus \sigma_e(T)$ such that $ind(T - \lambda) = -1$ for $\lambda \in \Omega,$ then there exists a conjugate analytic vector-valued function $k_\lambda\in \mathcal H$ on $\Omega$ such that for $\lambda \in \Omega,$ $k_\lambda \ne 0$ and $(T-\lambda )^*k_\lambda=0.$
\end{lemma}

\begin{lemma}\label{CDExtension}
Suppose that $T\in \mathcal L(\mathcal H),$ $0\in \sigma(T) \setminus \sigma_e(T),$ and $0$ is not an eigenvalue of $T.$
If $S\in \mathcal L(\mathcal H)$ and $TS=ST,$ then there exists $u_0\in \sigma(S)$ and $\delta> 0$ such that

(1) $\D(0, \delta) \subset \sigma(T) \setminus \sigma_e(T);$

(2) there exists $\delta_0 > 0$ (depending on $\delta$) and $K_u\in \mathcal H$ such that  $(S-u)^* K_u = 0,$ $\|K_u\| = 1,$  and $T^*K_{u_0} = 0$ for $u\in \D(u_0,\delta_0);$ and

(3) For $0 < |u_1 - u_0| < \delta_0,$ there exists $0 < \delta_1 < \min (|u_1 - u_0|, \delta_0 - |u_1 - u_0|)$ such that $K_u$ in (2) can be chosen as conjugate analytic on $\D(u_1, \delta_1).$   
 \end{lemma}  

\begin{proof}
From Lemma \ref{CowenDouglas}, there exist linearly independent conjugate analytic vector-valued functions $\{k_\lambda^j\}_{j=1}^N \subset \mathcal H$ on $\D(0, \delta')\subset \sigma(T) \setminus \sigma_e(T)$ such that   
 \[
 \ T^*k_\lambda^j = \bar \lambda k_\lambda^j, ~j = 1,2,...,N 
 \]
where $N = dim (ker(T-\lambda)^*)$. Since $T \mathcal H$ is an invariant subspace of $S$,  $M_\lambda = ker(T-\lambda)^*$ is an invariant subspace of $S^*$ and $dim(M_\lambda) = N$ for $\lambda\in \D(0, \delta).$ Let
 \[
 \ S^*k_\lambda^j = \sum_{l=1}^N \overline{p_{jl}(\lambda)} k_\lambda^l.
 \]
The matrix $((k_0^i, k_0^j))_{n\times n}$ is invertible. Set the following matrices
 \[
 \ S(\lambda) = ((Sk_0^i, k_\lambda^j))_{n\times n}, ~ K(\lambda) = ((k_0^i, k_\lambda^j))_{n\times n}, \text{ and } P(\lambda) = (p_{jl}(\lambda))_{n\times n}. 
 \]
Then
 \[
 \ S(\lambda) = K(\lambda) P(\lambda).
 \]
We choose $\delta$ small enough such that $K(\lambda)$ is invertible on $\D(0, \delta).$ Hence, $K(\lambda)$ and  $S(\lambda)$ are analytic on $\D(0, \delta)$. Therefore,
 \[
 \ P(\lambda) =  K^{-1}(\lambda) S(\lambda).
 \]
is analytic on $\D(0, \delta)$. Define
 \[
 \ p(\lambda, u) = det (P(\lambda) - uI),
 \]
then there exists $u_0$ such that $p(0, u_0) = 0$ and there exists a unit vector $K_{u_0}$ that is a linear combination of $\{k_o^j\}$ such that $(S-u_0)^*K_{u_0} = 0.$ Using Puiseux's Expansion, for $\delta < \delta',$ 
 we find $\delta_0 > 0$ such that the Riemann surface
 \[
 \ \{(\lambda, u):~ \lambda\in \D(0,\delta),~ u\in \D(u_0,\delta_0), ~ p(\lambda, u) = 0 \}
 \] 
 has a solution $\lambda = \lambda(z) = z^a$ and $u = u(z) = u_0 + \sum_{n=0}^\i c_n  z^{b+n}$ for some positive integers $a$ and $b$ near $(0, u_0).$ We select $\delta_2 > 0$ small enough such that $\frac{\partial \lambda (z)}{\partial z} \ne 0$ and $\frac{\partial u (z)}{\partial z} \ne 0$ for $z\in \D(0,\delta_2) \setminus \{0\}$ and select $\delta_0$ small enough such that $\D(u_0,\delta_0) \subset u(\D(0,\delta_2)).$
 Therefore, for $0 < |u_1 - u_0| < \delta_0$ and $0 < \delta_1 < \min (|u_1 - u_0|, \delta_0 - |u_1 - u_0|),$ there exists an analytic function $\lambda (u)$ from $\D(u_1,\delta_1)$ into $\D(0,\delta)$ such that $p_0(\lambda (u), u) = 0$ for $u \in \D(u_1,\delta_1).$ (1), (2), and (3) follow.
\end{proof}

\begin{lemma}\label{RInteriorLemma}
Let $S\in \mathcal L(\mathcal  H)$ be a pure subnormal operator and let  $k_\lambda\in \mathcal H$ be a non-zero conjugate analytic vector-valued function on $\D(0,\delta)$ such that $S^*k_\lambda = \bar \lambda k_\lambda$. Then $\D(0,\delta)\subset \mathcal R_S, ~\gamma-a.a.,$ where $\mathcal R_S$ is defined as in \eqref{RSDef}.
\end{lemma}

\begin{proof}
Suppose that $E\subset \D(0,\delta)\cap \setminus \mathcal R_S$ is a compact subset with $\gamma(E) > 0$. From \eqref{REEEqualsHIR}, Theorem \ref{MAEqualREE}, and Proposition \ref{MAProp} (5), we find a function $f(z)$ that is bounded and analytic on $E^c$ such that $f(\infty) = 0,$ $f'(\infty) \ne 0,$ and
 \[
 \ f(z), ~\dfrac{f(z) - f(\lambda)}{z - \lambda} \in M_a(S), ~ \lambda \in \D(0,\delta) \setminus E.
 \]
Hence,
 \[
 \ (fk_0, k_\lambda) - f(\lambda)(k_0, k_\lambda) = \left(\dfrac{f(z) - f(\lambda)}{z - \lambda}k_0, (S-\lambda I)^*k_\lambda\right) = 0.
 \]
Therefore, $f$ can be extended as a bounded and analytic function on $\mathbb C,$ which implies $f = 0.$ This is a contradiction.
 
\end{proof}

\begin{corollary}\label{RInteriorCor}
Let $S\in \mathcal L(\mathcal H)$ be a pure subnormal operator and let $\mu$ be the scalar-valued spectral measure for $N=mne(S).$
Then $\sigma(S) \setminus (\sigma_e(S) \cap \text{spt}\mu) \subset \mathcal R_S, ~\gamma-a.a.$.
\end{corollary}

\begin{proof} 
Let $\mathbb D(0,\delta) \subset \sigma(S) \setminus (\sigma_e(S) \cap \text{spt}\mu).$ Then there exists $T\in \mathcal L(\mathcal H)$ such that $TS = I$ and $S^*k = 0$ for some $k\ne 0.$ The vector 
\[
\ K_\lambda = (1 - \bar \lambda T^*)^{-1}k
\]
is co-analytic on $\mathbb D(0,\delta)$ for $\delta < \|T\|^{-1}$ and $S^*K_\lambda = \bar \lambda K_\lambda.$
	The proof now follows from Lemma \ref{RInteriorLemma}. 
\end{proof}

\begin{proof} (Theorem \ref{theoremA} (2)):
Suppose that $f\in M_a(S)$, then  by Corollary \ref{RInteriorCor} and Proposition \ref{MAProp} (5), $f$ is analytic on $\sigma(S) \setminus \sigma_e(S)$ and
\[
 \ \dfrac{f(z) - f(\lambda)}{z - \lambda} \in M_a(S), ~ \lambda \in \sigma(S) \setminus \sigma_e(S).
 \]
Since $M_{S,f}$  is Fredholm, $f$ has only finite many zeros in $\sigma(S)\setminus \sigma_e(S)$. Let $f(z) = p(z) f_0(z)$, where $p(z)$ has only zeros in $\sigma(S)\setminus \sigma_e(S)$, $f_0 \in M_a(S)$ has no zeros in $\sigma(S)\setminus \sigma_e(S)$, and $M_{S,f_0}$ is Fredholm. 

Set $\Delta = \{z:~ f_0(z) \ne 0\}$ and $\mathcal H_\Delta = \{H\in \mathcal H:~ H(z) = 0~\mu_\Delta-a.a.\}.$ Then $ker(M_{S,f_0}) = \mathcal H_\Delta ,$ $\dim(\mathcal H_\Delta) < \infty,$ and $S\mathcal H_\Delta \subset \mathcal H_\Delta.$ Therefore, $S|_{\mathcal H_\Delta}$ is a normal operator. Hence, $\mathcal H_\Delta = 0$ since $S$ is pure. That is, $0$ is not an eigenvalue of $M_{S,f_0}.$

We want to show that
$ M_{S,f_0} \mathcal H = \mathcal H.$
Otherwise, we can find $u_0$ and $ \delta_0, ~ \delta > 0$ satisfying (1), (2), and (3) of Lemma \ref{CDExtension} and we see that, by Lemma \ref{RInteriorLemma},  $\D(u_0, \delta_0) \subset \mathcal R_S,~ \gamma-a.a.,$ hence, $f_0$ is analytic on $\D(u_0, \delta_0).$ There exists $g_0 \in M_a(S)$ such that $f_0(z) - f_0(u_0) = (z - u_0) g_0$. By Lemma \ref{CDExtension} (2), there exists $k\in \mathcal H$ and $k \ne 0$ such that $S^* k = \bar u_0 k$ and $M_{S,f_0}^* k = 0$. Hence, $f_0(u_0) = 0$. Therefore, $S-u_0I$ is Fredholm and $u_0\in \sigma(S)\setminus \sigma_e(S).$ This is a contradiction since $f_0$ has no zeros in $\sigma(S)\setminus \sigma_e(S).$ Thus, $M_{S,f_0}$ is invertible.   
\end{proof} 

For $f\in L^\i (\mu),$ denote $\sigma(f)$ the essential range of $f$ in $L^\i (\mu).$ 

\begin{corollary} \label{corollaryB}
Let $S\in \mathcal L(\mathcal H)$ be a subnormal operator and let $\mu$ be the scalar-valued spectral measure for $N=mne(S).$ 
Then for $f \in R^\i(\sigma(S),\mu),$
\[
\ \sigma (M_{S,f}) = \sigma (f) \cup cl(\rho_S(f), E_S) 
\]
and
\[
\ \sigma_e (M_{S,f}) = \sigma (f) \cup  cl_e(\rho_S(f), E_S).
\]
In particular, if $S$ is pure, then
\[
\ \sigma (M_{S,f}) = cl(\rho_S(f), E_S) 
\text{ and } \sigma_e (M_{S,f}) =  cl_e(\rho_S(f), E_S).
\]
\end{corollary}

\begin{proof}
Let $\Delta_0$ and $\Delta_1$ be as in \eqref{MDecompEq}. Then $f$ is invertible in $M(S)$ if and only if $f$ is invertible	 in $L^\i (\mu)$ and $f_{\Delta_1}$ is invertible	 in $M(S_{\Delta_1}).$ Applying Corollary \ref{corollaryA} and Theorem \ref{InvInMA}, we prove the corollary.
\end{proof}

\section{\textbf{Spectral Mapping Theorems for Irreducible Subnormal Operators}}

Notice that the construction of $S$ in the proof of Corollary \ref{ExampleCor} indicates that for $S$ and $f$ in the corollary, $S = \oplus_{n=1}^\i S_n,$ where $S_n$ is irreducible, and \eqref{DudziakEq1Fail} holds. Our next result shows this phenomenon disappears when $S$ is irreducible. As a result, we show that the second open question in \cite[page 386]{dud84} has an affirmative answer for irreducible subnormal operators.
In fact, Theorem \ref{theoremC} also holds for $S = \oplus_{n=1}^m S_n,$ where $m < \i$ and $S_n$ is irreducible.

\begin{theorem} \label{theoremC}
Let $S\in \mathcal L(\mathcal H)$ be an irreducible subnormal operator and let $\mu$ be the scalar-valued spectral measure for $N=mne(S).$ Then for $f\in R^\i(\sigma (S), \mu),$ the following properties hold:

(1) $M_{S,f}$ is invertible if and only if $f$ is invertible in $R^\i(\sigma (S), \mu).$

(2) $M_{S,f}$ is Fredholm if and only if $f = p f_0$, where $p$ is a polynomial with only zeros $z_1,z_2,...,z_m\in \sigma(S)\setminus \sigma_e(S)$, multiplicity of $z_i$ is $n_i$, and $f_0$ is invertible in $R^\i(\sigma (S), \mu).$
In this case,
 \[
 \ \text{ind}(M_{S,f}) = \sum_{i=1}^m n_i \text{ind}(S-z_i).
 \]	
\end{theorem}

Consequently, we have the following spectral mapping theorem for irreducible subnormal operators.

\begin{corollary} \label{corollaryC}
Let $S\in \mathcal L(\mathcal H)$ be an irreducible subnormal operator and let $\mu$ be the scalar-valued spectral measure for $N=mne(S).$ Then for $f\in R^\i(\sigma(S),\mu),$
\[
\ \sigma (M_{S,f}) = cl(\rho_S(f), E_S^\i)  
\]
and 
\[
\ \sigma_e (M_{S,f}) = cl_e(\rho_S(f), E_S^\i).
\]	
\end{corollary}

The following is an important property for an irreducible subnormal operator. Recall that the removable set $\mathcal R_S$ is defined as in \eqref{RSDef}. Set $\mathcal F_S = \C \setminus \mathcal R_S.$

\begin{theorem}\label{IrreducibleForS} 
Let $S\in \mathcal L(\mathcal H)$ be an irreducible subnormal operator and let $\mu$ be the scalar-valued spectral measure for $N=mne(S).$ Then for $\lambda \in \sigma (S),$
 \[
 \ \underset{\delta \rightarrow 0}{\underline{\lim}}\dfrac{\gamma(\D(\lambda, \delta)\setminus \mathcal F_S)}{\delta} = \underset{\delta \rightarrow 0}{\underline{\lim}}\dfrac{\gamma(\D(\lambda, \delta)\cap \mathcal R_S)}{\delta} > 0.
 \] 
\end{theorem}

To prove Theorem \ref{IrreducibleForS}, we need the following lemma, which combines Thomson's coloring scheme (see \cite{thomson} and \cite[Section 2]{y19}) and \cite[Lemma 2]{y19}.

\begin{lemma} \label{LightRoute}
For $\epsilon > 0$, there exists $\epsilon_0 > 0$ (depending on $\epsilon$) such that if a subset $E \subset \mathbb D$ and $\gamma (\mathbb D \setminus E) < \epsilon_0$,
then there is a sequence of simply connected regions $\{G_n\}_{n \ge m}$ for some integer $m \ge 1$ satisfying:
\[
\ \mathbb D \left (0, \frac 12 \right ) \subset G_n \subset G_{n+1} \subset \mathbb D 
\]
and for $\lambda\in \partial G_n,$ we have
\[
\ \mathbb D (\lambda, n^2 2^{-n}) \subset G_{n+1} \text{ and }\gamma (\mathbb D (\lambda, 2^{-n}) \setminus E) < \epsilon 2^{-n}.
\]
\end{lemma}

Let $f$ be an analytic function outside the disc $\overline{\mathbb D(a, \delta)}$ satisfying the condition $f(\infty ) = 0.$ We consider the Laurent expansion of $f$ centered at $a$,
 \[
 \ f(z) = \sum_{n=1}^\infty \dfrac{c_n(f,a)}{(z-a)^n}.
 \]
$c_1(f) := c_1(f,a)$ does not depend on the choice of $a$, while $c_2 (f,a)$ depends on $a$. 

Let $\delta > 0$. We say that 
$\{\varphi_{ij},S_{ij}, \delta\}$ is a smooth partition of unity subordinated to $\{2S_{ij}\},$ 
 if the following assumptions hold:

(i)  $S_{ij}$ is a square with vertices   $(i\delta,j\delta),~((i+1)\delta,j\delta),~(i\delta,(j+1)\delta),$ and $((i+1)\delta,(j+1)\delta);$

(ii) $c_{ij}$ is the center of $S_{ij}$ and $\varphi_{ij}$ is a $C^\infty$ smooth function supported in $\mathbb D(c_{ij}, \delta) \subset 2S_{ij}$ and with values in $[0,1]$;

(iii) 
 \[
 \ \|\bar\partial \varphi_{ij} \| \lesssim \frac{1}{\delta},~ \sum \varphi_{ij} = 1.
 \]
(See \cite[VIII.7]{gamelin} for details).

Let  $\mathcal Q$ be a set with $\gamma(\mathcal Q) = 0$.
Let $f(z)$ be a function defined 
on $\D(\lambda, \delta_0)\setminus \mathcal Q$ for some $\delta_0 > 0.$ The function $f$ is $\gamma$-continuous at $\lambda$ if for all $\epsilon > 0,$
\begin{eqnarray} \label{GCDef}  
 \  \lim_{\delta \rightarrow 0} \dfrac{\gamma(\D(\lambda, \delta) \cap \{|f(z) - f(\lambda)| > \epsilon\})} {\delta}= 0.
\end{eqnarray}

\begin{proof} (Theorem \ref{IrreducibleForS}):
Fix $\epsilon < \frac{1}{36C_T},$ where $C_T$ is from \eqref{SAAC}. If there exists $\lambda \in \sigma(S)$ such that 
 \[
 \ \underset{\delta \rightarrow 0}{\underline{\lim}}\dfrac{\gamma(\D(\lambda, \delta)\setminus \mathcal F_S)}{\delta} = 0,
 \]
then there exists $\delta_0 > 0$ such that
 \begin{eqnarray}\label{lemma3Eq1}
 \ \gamma(\D(\lambda, \delta_0)\setminus \mathcal F_S) < \epsilon_0 \delta_0,  
 \end{eqnarray}
where $\epsilon$ and $\epsilon_0$ are as in Lemma \ref{LightRoute}, $\mu (\sigma(S) \setminus \overline{\D(\lambda, \delta_0)}) > 0,$ and by Corollary \ref{RInteriorCor}, we assume $\mu (\D(\lambda, \delta)) > 0$ for $\delta \le \delta_0.$ Without loss of generality, we assume that $\lambda = 0$ and $\delta_0 = 1$. Let $\{G_n\}_{n \ge m}$ be the sequence of simply connected regions as in Lemma \ref{LightRoute}. 

Let $f_n = \chi_{G_n}.$ Then $f_n \in H^\infty (\mathbb C \setminus \partial G_n)$. 
Let $\{\varphi_{ij},S_{ij}, 2^{-n}\}$ be a smooth partition of unity as above. Then
 \begin{eqnarray}\label{SVEq}
 \ f_n = \sum_{\mathbb D(c_{ij}, 2^{-n}) \cap \partial G_n \ne \emptyset} (f_{ij}:= T_{\varphi_{ij}}f_n) = \sum_{\mathbb D(c_{ij}, 2^{-n}) \cap \partial G_n \ne \emptyset} (f_{ij} - g_{ij} - h_{ij}) + \hat f_n. 
 \end{eqnarray}
where 
 \[
 \ \hat f_n = \sum_{\mathbb D(c_{ij}, 2^{-n}) \cap \partial G_n \ne \emptyset} (g_{ij} + h_{ij}).
 \]
To apply the standard Vitushkin scheme, we need to construct functions $g_{ij}, h_{ij} \in H^\infty(\mathcal R_S)$ that are bounded and analytic off $3S_{ij}$ and the following properties hold:
 \begin{eqnarray}\label{gProperty}
 \ \|g_{ij}\|_{L^\infty(\area_{\mathcal R_S})} \lesssim 1, ~ g_{ij}(\infty) = 0, ~ c_1(g_{ij}) = c_1(f_{ij})
\end{eqnarray}
and
\begin{eqnarray}\label{hProperty}
 \ \begin{aligned}
 \ & \|h_{ij}\|_{L^\infty(\area_{\mathcal R_S})} \lesssim 1, ~ h_{ij}(\infty) = c_1(h_{ij}) = 0, \\
 \ & c_2(h_{ij}, c_{ij}) = c_2(f_{ij} - g_{ij}, c_{ij}).
 \ \end{aligned}
\end{eqnarray}

Assuming $g_{ij}, h_{ij}$ have been constructed and satisfy \eqref{gProperty} and \eqref{hProperty}, we get 
 \[
 \ |f_{ij}(z) - g_{ij}(z) - h_{ij}(z) | \lesssim \min \left (1, \dfrac{\frac{1}{2^{3n}}}{|z - c_{ij}|^3} \right ).
 \]
Therefore, using
the standard Vitushkin scheme, we see that there exists $\hat f\in H^\infty(\mathcal R_S)$ and a subsequence $\{\hat f_{n_k}\} \subset H^\infty(\mathcal R_S)$ such that $\hat f_{n_k}$ converges to $\hat f$ in weak-star topology in $L^\infty (\area_{\mathcal R_S} )$. Clearly, $\hat f(z) = \chi_G$ for $z\in \mathbb C \setminus \partial G$, where $G= \cup G_n$ is a simply connected region.
For $\lambda \in \mathcal R_S \cap \partial G,~\gamma-a.a.,$ $\underset{\delta\rightarrow 0}{\underline \lim}\frac{\gamma(\D(\lambda, \delta)\cap G)}{\delta} > 0$ since $G$ is connected. 
Using \eqref{SAAC} and the fact that $\lim_{\delta\rightarrow 0} \frac{\gamma(\D(\lambda, \delta)\setminus \mathcal R_S)}{\delta} = 0$ by \cite[Corollary 5.7]{y23}, we get
\[
\ \underset{\delta\rightarrow 0}{\underline{\lim}} \dfrac{\gamma(\D(\lambda, \delta)\cap G\cap \mathcal R_S)}{\delta} \ge \dfrac{1}{C_T} \underset{\delta\rightarrow 0}{\underline \lim}\frac{\gamma(\D(\lambda, \delta)\cap G)}{\delta} - \lim_{\delta\rightarrow 0} \dfrac{\gamma(\D(\lambda, \delta)\setminus \mathcal R_S)}{\delta} > 0.
\]
Using \cite[Lemma 7.3 (1)]{y23}, we see that $\hat f(z)$ is $\gamma$-continuous at $\lambda \in \mathcal R_S \cap \partial G,~\gamma-a.a.$ (see \eqref{GCDef}) and conclude that
\[
\ \hat f(z) = \chi_{\mathcal R_S \cap \overline G},~\area_{\mathcal R_S}-a.a..
\]
By Theorem \ref{MAEqualREE}, there exists $\Delta$ with $G \subset \Delta \subset \overline G$ such that $\chi_{\Delta} \in M_a(S)$ and $\hat f = \rho_S(\chi_{\Delta})$. Since $\mu (\D(\lambda, \delta)) > 0$ for $\delta \le \delta_0,$ $\chi_{\Delta}$ is a non-trivial characteristic function, which contradicts the assumption that $S$ is irreducible.  

It remains to construct $g_{ij}, h_{ij}$ satisfying \eqref{gProperty} and \eqref{hProperty}. For $l = 2^{-n}$ and $ u\in \overline {\mathbb D(c_{ij}, l)}\cap \partial G_n \ne \emptyset$, let $V$ be the square centered at $u$ whose edges are parallel to the coordinate axes and length equals $l.$ Then $V\subset 3S_{ij}$ such that, from Lemma \ref{LightRoute}, we have
 \[
 \ \gamma(V \setminus \mathcal F_S) < \epsilon l .
 \]
We divide $V$ into 9 equal small squares. Let $V_d$ (with center $c_d$)and $V_u$ (with center $c_u$) be the left bottom and upper squares, respectively. Then, using \eqref{SAAC}, we conclude that
 \[
 \ \gamma (V_d\cap \mathcal F_S) \ge \dfrac{1}{C_T} \gamma(V_d) - \gamma (V_d\setminus \mathcal F_S) \ge \dfrac{1-12C_T\epsilon}{12C_T}l \ge \dfrac{1}{18C_T}l. 
 \]
Similarly, $\gamma (V_u\cap \mathcal F_S) \ge \frac{1}{18C_T}l$.
Using \eqref{GammaEqualGammaP},  we find $\eta_d\in M_0^+(V_d\cap \mathcal F_S)$ and $\eta_u\in M_0^+(V_u\cap \mathcal F_S)$ such that $\|\eta_d\|= \gamma (V_d\cap \mathcal F_S),$ $\|\eta_u\|= \gamma (V_u\cap \mathcal F_S),$ $\|h_d:= \CT \eta_d \| \lesssim 1,$ and  $\|h_u:= \CT \eta_u \| \lesssim 1.$  Let
 \[
 \ h_0 = \dfrac{\|\eta_u\|}{l}h_d - \dfrac{\|\eta_d\|}{l}h_u. 
 \]
Then $c_1 (h_0) = 0$ and
 \[
 \ \begin{aligned}
 \  |c_2(h_0, c_{ij}) | = & \dfrac{1}{l} \left | \int\int (w-z)d\eta_u(w)d\eta_d(z)\right | \\
 \ \ge & \dfrac{1}{l} \int \int Im(w-z)d\eta_u(w)d\eta_d(z) \\
 \ \ge & \dfrac{l^2}{972C_T^2}. 
 \end{aligned}
 \]
Set 
 \[
 \ g_{ij} = \dfrac{h_d}{c_1(h_d)} c_1(f_{ij}) \text{ and } h_{ij} = \dfrac{h_0}{c_2(h_0, c_{ij})} c_2(f_{ij} - g_{ij}, c_{ij}). 
 \]
It is easy to verify that $g_{ij}$ and $h_{ij}$ satisfy \eqref{gProperty} and \eqref{hProperty}.	
\end{proof}

\begin{lemma}\label{InvLemma}
Let $S\in \mathcal L(\mathcal H)$ be an irreducible subnormal operator, let $\mu$ be the scalar-valued spectral measure for $N=mne(S),$ and let $f\in R^\infty(\sigma(S), \mu).$ If there exists $\epsilon_f > 0$ such that
\[
\ |\rho_S(f)(z)| \ge \epsilon_f,~\area_{\mathcal R_S}-a.a., 
\]	
then $f$ is invertible in $R^\infty(\sigma(S), \mu).$
\end{lemma}

\begin{proof}
Let $\mathcal R_\infty$ be the removable set for $R^\infty(\sigma(S), \mu).$ From [Lemma 7.3 (1)]\cite{y23} , we see that $\rho_S(f)(z)$ is $\gamma$-continuous at $\lambda\in \mathcal R_\infty~\gamma-a.a..$	For $\epsilon > 0,$ define
\[
\ A_\epsilon = \{|\rho_S(f)(z) - \rho_S(f)(\lambda)| > \epsilon\}.
\]
Using \eqref{SAAC}, we have
\[
\ \gamma(\mathbb D(\lambda, \delta)\cap A_\epsilon^c \cap \mathcal R_S) + \gamma(\mathbb D(\lambda, \delta)\cap A_\epsilon)\ge \dfrac{1}{C_T} \gamma(\mathbb D(\lambda, \delta) \cap \mathcal R_S).
\]
Applying Theorem \ref{IrreducibleForS} and $\gamma$-continuity of $\rho_S(f)(z)$ at $\lambda,$ we get
\[
\ \underset{\delta\rightarrow 0}{\underline{\lim}} \dfrac{\gamma(\mathbb D(\lambda, \delta)\cap A_\epsilon^c \cap \mathcal R_S)}{\delta} > 0.
\]
Together with the fact that $\lim_{\delta\rightarrow 0} \frac{\gamma(\D(\lambda, \delta)\setminus \mathcal R_S)}{\delta} = 0$ by \cite[Corollary 5.7]{y23} and \eqref{AreaGammaEq}, we find a sequence $\{\lambda_n\} \subset \mathcal R_S$ such that $|\rho_S(f)(\lambda_n)| \ge \epsilon_f$ and $\rho_S(f)(\lambda_n)\rightarrow \rho_S(f)(\lambda).$ Therefore, we infer that
\[
\ |\rho_S(f)(z)| \ge \epsilon_f,~\area_{\mathcal R_\infty}-a.a..
\]
Now the proof follows from \cite[Corollary 5.9]{y23} and Theorem \ref{InvTheoremY23}.
\end{proof}

Combining Theorem \ref{theoremA}, Theorem \ref{InvInMA}, and Lemma \ref{InvLemma}, we finish the proof of Theorem \ref{theoremC}.

\bigskip

{\bf Acknowledgments.} 
The author would like to thank Professor John M\raise.45ex\hbox{c}Carthy for carefully reading through the manuscript and providing many useful comments.

\bigskip

\bibliographystyle{amsplain}

\begin{thebibliography}{99}


\bibitem{a82} S. Axler, \textit{Multiplication operators on Bergman spaces}, J. reine anger.  Math. {\bf 336} (1982). 26--44.

\bibitem{acm82} S. Axler, J. Conway, and G. McDonald \textit{Toeplitz operators on Bergman spaces}, Canadian J. Math. {\bf 334} 2(1982). 466--483.

			
\bibitem{cha74} J. Chaumat, \textit{Adherence faible \'etoile d'alg\'ebra de fractions rationelle}, Ann. Inst. Fourier Grenoble,  {\bf 24} (1974), 93--120.

\bibitem{c73} L. A. Coburn, \textit{Singular integral operators and Toeplitz operators on odd spheres}, Indiana Univ. Math. J. {\bf 23} (1973). 433--439.


\bibitem{c81} J. B. Conway, \textit{Subnormal operators}, Pitman, London, 1981.
  
 \bibitem{conway} J. B. Conway, \textit{The theory of subnormal operators}, Mathematical Survey and Monographs 36, 1991.
 
 \bibitem{co77} J B. Conway and R. F. Olin, \textit{A functional calculus for subnormal operators, II}, Mem. Amer. Math. Soc. (1977), no,184.


\bibitem{cd78} M.J. Cowen and R.G. Douglas, \textit{Complex Geometry and Operator Theory}, Acta. Math.  {\bf 141} (1978), 187--261.


\bibitem{dud84} James Dudziak, \textit{Spectral Mapping Theorems for Subnormal Operators}, Journal of Functional Analysis  {\bf 56} (1984), 360--387.

  
\bibitem{gamelin} T. W. Gamelin, \textit{Uniform algebras}, American Mathematical Society, Rhode Island, 1969.

\bibitem{j81} J. Janas, \textit{Toeplitz operators related to certain domains in C}, Studia Math.  {\bf 154} (1981), no. 3, 73--77.


\bibitem{m77} G. MCDonald, \textit{Fredholm properties of a class of Toeplitz operators on the ball}, IndianaUniv.Math.J.  {\bf 26} (1977), 567--576.



\bibitem{thomson} J. E. Thomson, \textit{Approximation in the mean by polynomials}, Ann. of Math.  {\bf 133} (1991), no. 3, 477--507.

\bibitem{Tol03} X. Tolsa, \textit{Painleve's  problem and the semiadditivity of analytic capacity}, Acta Math.  {\bf 51} (2003), no. 1, 105--149.

\bibitem{Tol14} X. Tolsa, \textit{Analytic capacity, the Cauchy transform, and non-homogeneous  Calder\'{o}n-Zygmund
theory}, Birkhauser/Springer, Cham, 2014.

\bibitem{y19} L. Yang, \textit{Bounded point evaluations for certain polynomial and rational modules}, Journal of Mathematical Analysis and Applications  {\bf 474} (2019), 219--241.

\bibitem{y23} L.Yang, \textit{Invertibility in weak-star closed algebras of analytic functions}, J. Funct. Anal., {\bf 285}(2023), No 11, Paper No. 110143, 32 pp.
\end{thebibliography}

\end{document}